\documentclass[12pt]{amsart}
\usepackage{enumerate,color,graphicx}
\usepackage{a4wide}
\usepackage{bm} 
\usepackage{newtxtext}
\usepackage{newtxmath} 
\allowdisplaybreaks

\let\pa\partial  
\let\na\nabla  
\let\eps\varepsilon  
\newcommand{\N}{{\mathbb N}}  
\newcommand{\R}{{\mathbb R}} 
\newcommand{\C}{{\mathbb C}}

\newcommand{\diver}{\operatorname{div}}

\newcommand{\T}{{\mathcal T}}
\newcommand{\M}{{\mathcal M}}
\newcommand{\Exp}{{\mathcal E}{\mathrm{xp}}}
\newcommand{\Log}{{\mathcal L}{\mathrm{og}}}
\newcommand{\ii}{{\mathrm i}}

\newtheorem{theorem}{Theorem}   
\newtheorem{lemma}[theorem]{Lemma}   
\newtheorem{proposition}[theorem]{Proposition}

\newtheorem{definition}{Definition}

\begin{document}  

\title[Spinorial quantum drift-diffusion equations]{
Formal derivation of quantum drift-diffusion equations \\ with spin-orbit interaction}

\author[L.~Barletti, P.~Holzinger, and A.~J\"ungel,]{Luigi Barletti, 
Philipp Holzinger, and Ansgar J\"ungel}

\address{Dipartimento di Matematica, Università di Firenze, Viale Morgagni 67/A, 
  50134 Firenze,	Italy}
\email{luigi.barletti@unifi.it} 

\address{Institute for Analysis and Scientific Computing, Vienna University of  
	Technology, Wiedner Hauptstra\ss e 8--10, 1040 Wien, Austria}
\email{philipp.holzinger@tuwien.ac.at} 

\address{Institute for Analysis and Scientific Computing, Vienna University of  
	Technology, Wiedner Hauptstra\ss e 8--10, 1040 Wien, Austria}
\email{juengel@tuwien.ac.at} 

\date{\today}

\thanks{The last two authors have been partially supported by the 
Austrian Science Fund (FWF), grants P30000, P33010, F65, and W1245. 
This work received funding from the European 
Research Council (ERC) under the European Union's Horizon 2020 research and 
innovation programme, ERC Advanced Grant NEUROMORPH, no.~101018153.}

\begin{abstract}
Quantum drift-diffusion equations for a two-dimensional electron gas with
spin-orbit interactions of Rashba type are formally derived from a collisional 
Wigner equation. The collisions are modeled by a Bhatnagar--Gross--Krook-type 
operator describing the relaxation of the electron gas to a local equilibrium
that is given by the quantum maximum entropy principle. Because of
non-commutativity properties of the operators, 
the standard diffusion scaling cannot be used
in this context, and a hydrodynamic time scaling is required.
A Chapman--Enskog procedure leads, up to first order in the relaxation time, 
to a system of
nonlocal quantum drift-diffusion equations for the charge density and
spin vector densities. Local equations including the Bohm potential
are obtained in the semiclassical
expansion up to second order in the scaled Planck constant.
The main novelty of this work is that all spin components are considered, 
while previous models only consider special spin directions.
\end{abstract}

\keywords{Wigner--Boltzmann equation, diffusion limit, spin-orbit interaction,
quantum maximum entropy principle, semiclassical model, Bohm potential.}  
 
\subjclass[2010]{35K55, 35Q40, 35Q81, 82B10.}  

\maketitle


\section{Introduction}

Spintronics exploits the electron spin as a further degree of freedom in
semiconductor materials. The objective of spintronics is to develop fast, high-capacity,
and low-power information and communication devices. The design of spintronic
structures is accelerated by numerical simulations that optimize the device
properties. To achieve efficient but physically accurate simulations,
macroscopic spin models, also including quantum features, are needed.

In the literature, usually simplified models are considered.
A simple approach is to consider specific directions of the spin vector,
for instance the spin-up and spin-down electron densities or, equivalently, the 
total density and the spin polarization \cite{ZFD02}. 
A more complete picture is obtained by taking into account 
the complete spin vector and not 
only its projection on a given direction. Such models have four variables: the
charge density and the densities of the three spin components \cite{ElH14,PoNe11}. 
Quantum corrections have been included in the former approach in \cite{BaMe10}, 
leading to spinorial quantum drift-diffusion equations for the spin-up and spin-down 
densities, while the works \cite{BHJ21,PoNe11} are concerned with the derivation of 
full spin-vector models but without quantum corrections. 
Up to our knowledge, no drift-diffusion models with a full spin structure 
{\em and} quantum corrections have been derived in the literature so far.
In this paper, we fill this gap by deriving spinorial quantum drift-diffusion
equations for a two-dimensional electron gas from a collisional von Neumann
equation. 

\subsection{Setting}

We consider an electron gas confined in an asymmetric two-dimensional
potential well. Then the electrons experience a spin-orbit interaction of Rashba type
\cite{ByRa84}. The Rashba effect is a momentum-based splitting of spin bands, which
comes from the combined effect of spin-orbit interaction and an asymmetry of the 
potential. It manifests as an effective magnetic field orthogonal to the confinement
direction and the electron motion; see Figure \ref{fig}. 
The spin orientation can be indirectly controlled
by the gate voltage, which deviates the electrons, thus changing the direction
of the effective magnetic field. We refer to the review \cite{ZFD04} for more details.

\begin{figure}[h]
\begin{center}
\includegraphics[width=120mm]{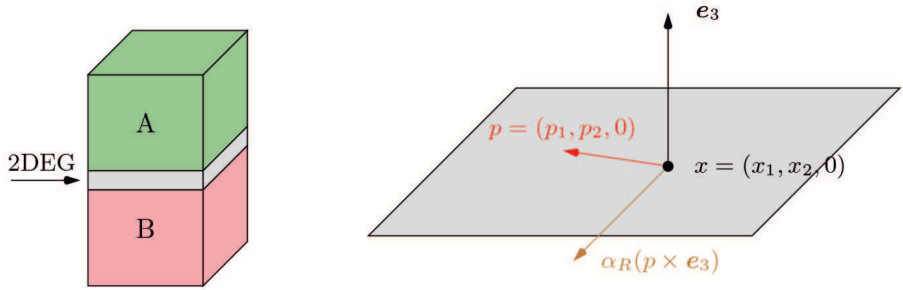} 
\caption{Left: A two-dimensional electron gas (2DEG) is confined between two
different semiconductor materials A and B (for instance, InAlAs and InGaAs). 
Right: The electrons of the 2DEG experience
an effective magnetic field $\alpha_R(p\times \bm{e}_3)$ orthogonal to both
the electron momentum $p$ and the confinement direction $\bm{e}_3$, where
$\alpha_R>0$ and $\bm{e}_3=(0,0,1)^T$.}
\label{fig}
\end{center}
\end{figure}

The motion of the confined electrons in the $(x_1,x_2)$-plane is governed
by the (scaled) von Neumann equation for the density operator $\hat\rho(t)$,
\begin{equation}\label{1.vne}
  \ii\eps\tau_0\pa_t\hat\rho = [\mathcal{H}_\eps,\hat\rho]
	:= \mathcal{H}_\eps\hat\rho - \hat\rho\mathcal{H}_\eps, \quad t>0,
\end{equation}
where the (scaled) Hamiltonian $\mathcal{H}_\eps$ is the sum of the kinetic energy, 
potential energy, and spin-orbit interaction,
\begin{equation}\label{1.H}
  \mathcal{H}_\eps = \bigg(-\frac{\eps^2}{2}\Delta + V(x)\bigg)\sigma_0
	- \eps^2\alpha\begin{pmatrix} 0 & \ii\pa_{x_2}-\pa_{x_1} \\
	\ii\pa_{x_2}+\pa_{x_1} & 0 \end{pmatrix}.
\end{equation}
Here, the function $V(x)$ is the electric (gate) potential, $\sigma_0\in\C^{2\times 2}$
is the identity matrix, $\eps>0$ is the scaled Planck constant,
$\tau_0$ is a scaled time, and $\alpha>0$ is the scaled Rashba constant.
We refer to Section \ref{sec.scal} for details on the scaling.

Our derivation is based on the phase-space formulation using the Wigner transform
$W:=\mathcal{W}(\hat\rho)$, defined in \eqref{2.wigner} below. 
Then equation \eqref{1.vne}
transforms to the Wigner equation (see Lemma \ref{lem.w})
$$
  \tau_0\pa_t W + \T W = 0, \quad t>0,
$$
where $\T W = (\ii/\eps)(H_\eps\# W - W\# H_\eps)$ is the transport operator, 
$H_\eps = \mathcal{W}(\mathcal{H}_\eps)$ is the Hamiltonian symbol,
and $\#$ denotes the Moyal product defined in \eqref{def.moyal} below. To derive
diffusion equations, we introduce a collision term of Bhatnagar--Gross--Krook 
(BGK) type:
\begin{equation}\label{1.swe}
  \tau_0\pa_t W + \T W = \frac{1}{\tau}(\M(N)-W), \quad t>0,
\end{equation}
where $\tau>0$ is the (scaled) relaxation time, 
$\M(N)=\mathcal{W}(\exp(-\mathcal{H}_\eps+\tilde{A}))$ is the so-called 
quantum Maxwellian, which (formally) minimizes the quantum free energy
under the constraint of a given density matrix $N=\langle W\rangle:=\int_{\R^2}Wdp$,
and $\tilde{A}$ is the associated Lagrange multiplier;
see Section \ref{sec.QM} for details.

If the time scale $\tau_0$ is of the same order as the (scaled) relaxation time
$\tau$, we obtain a diffusive scaling. The usual way to derive a macroscopic model
is the Chapman--Enskog expansion. Let $W_\tau$ be a solution to \eqref{1.swe}
with $\tau_0=\tau$ and write $W_\tau=W_0+\tau G_\tau$ for two functions
$W_0$ and $G_\tau$. The formal limit $\tau\to 0$ in \eqref{1.swe} determines the
first function, $W_0=\M(N)$. 
The expansion $W_\tau=W_0+\tau G_\tau$ in fact defines $G_\tau$.
Inserting this expansion into \eqref{1.swe}, dividing by $\tau$, and performing
the formal limit $\tau\to 0$ leads to $G:=\lim_{\tau\to 0}G_\tau=-\T\M(N)$.
The last step is to integrate \eqref{1.swe} with respect to $p$,
$$
  \pa_t\langle W_\tau\rangle + \frac{1}{\tau}\langle\T\M(N)\rangle
	+ \langle\T G_\tau\rangle = 0,
$$
and to pass to the limit $\tau\to 0$. In the classical situation, 
$\T\M(N)$ is an odd function in $p$ and therefore, its integral with respect
to $p$ vanishes. Physically, this means that the equilibrium state $\M(N)$
has a vanishing diffusion current. The limit $\tau\to 0$ then leads to the
macroscopic model $\pa_t N - \langle\T\T\M(N)\rangle = 0$, where
$\langle\T\T\M(N)\rangle$ is a drift-diffusion term. In the present case, however,
it turns out that generally $\langle\T\M(N)\rangle \neq 0$ (see Lemma \ref{lem.comm}).
This means that there is a residual current in the equilibrium state that is due
to the spin-orbit interaction. We show in Lemma \ref{lem.comm} that the condition
$\langle\T\M(N)\rangle \neq 0$ can be characterized by the non-commutativity
between the density matrix $N$ and the Lagrange multiplier $\tilde{A}$.

Therefore, we impose a hydrodynamic scaling, which is suitable for local equilibria
with non-vanishing currents. We stress the fact that the BGK collisions do not
conserve the current. The residual current $\T\M(N)$ is of quantum mechanical nature and
in our case, it is of order $\eps^3$ (see \eqref{3.TM} and \eqref{4.ntimesa}). 
We suppose that $\tau_0$ is of order
one, while $\tau\ll 1$. Then the expansion $W_\tau=\M(N)+\tau G_\tau$ in \eqref{1.swe}
leads to $G_\tau=-\pa_t\M(N)-\T\M(N)+O(\tau)$. We integrate \eqref{1.swe} with
respect to $p$, divide the equation by $\tau$, and insert the expansion
$W_\tau=\M(N) - \tau(\pa_t\M(N) + \T\M(N)) + O(\tau^2)$:
$$
  \pa_t N = -\langle\T\M(N)\rangle
	= -\langle\T\M(N)\rangle + \tau\langle\T\pa_t\M(N)\rangle
	+ \tau\langle\T\T\M(N)\rangle + O(\tau^2).
$$
Neglecting terms of order $O(\tau^2)$, we arrive at our diffusion equation,
with the diffusion contained in the term $\langle\T\T\M(N)\rangle$. 
The task is to compute the expressions on the right-hand side in terms of the
density matrix and related variables. 

Our key assumption is 
that the spin density is of order $\eps\ll 1$. Physically, this means that the
system is in a mixed state; the spin direction of the electrons is random, and a small
polarisation emerges from the average. Mathematically, this assumption simplifies 
the semiclassical expansion of the model. Indeed, the explicit computations
appear to be impractical when the spin density is of the same order as the
charge density.

\subsection{Main results}

Expressing the density matrix in terms 
of the Pauli basis $\sigma_0,\ldots,\sigma_3\in\C^{2\times 2}$ 
(see Section \ref{sec.phase}),
we write $N=n_0\sigma_0+\eps\bm{n}\cdot\bm{\sigma}$, where the coefficients 
are the charge density $n_0$ and the spin density $\eps\bm{n}$, 
and $\bm{n}\cdot\bm{\sigma}:=\sum_{j=1}^3n_j\sigma_j$. 
Similarly, we write the Lagrange multiplier matrix as
$\tilde{A}=\tilde{a}_0\sigma_0+\tilde{\bm{a}}\cdot\bm{\sigma}$. 
We prove in Section \ref{sec.QM} that actually $\tilde{\bm{a}}=\eps\bm{a}$ is
of order $\eps$. Moreover, we show that the Pauli components of the density matrix 
solve a system of nonlocal diffusion equations.

\begin{theorem}[Nonlocal quantum-spin model]\label{thm.full}
Let $W$ be a solution to the Wigner--Boltzmann equation \eqref{2.wbe} and set
$N=\langle W\rangle = n_0\sigma_0 + \eps\bm{n}\cdot\bm{\sigma}$. Let $\M(N)$
be the quantum Maxwellian defined in Theorem \ref{thm.QM}, $\tilde{A}
=\tilde{a}_0\sigma_0+\eps\bm{a}\cdot\bm{\sigma}$ be the matrix of Lagrange
multipliers, and $\bm{J}=\langle p\bm{\M}(N)\rangle$ be the full current density.
Then, at first order in $\tau$, the Pauli components of $N$ solve the following
equations:
\begin{align}
  \pa_t n_0 &= \tau\diver\big(n_0\na a_0 + n_0\na V + \eps^2\bm{n}\cdot\na\bm{a}\big)
	+ 2\alpha\eps^2\tau\na^\perp\cdot(\bm{n}\times\bm{a}), \label{1.n0} \\
	\pa_t\bm{n} &= -2\bm{n}\times\bm{a} + \tau\diver\bigg(n_0\na\bm{a} + \bm{n}\na a_0
	+ \bm{n}\na V + \frac{2}{\eps}\bm{J}^T\times\bm{a}\bigg) \label{1.n} \\
	&\phantom{xx}{}- 2\alpha\tau\Big(n_0\na^\perp\times\bm{a}
	+ \na^\perp(a_0+V)\times\bm{n} - \frac{2}{\eps}\big(\bm{a}\langle p^\perp\cdot
	\M(N)\rangle + \bm{J}^T\bm{a}^\perp\big)\Big) \nonumber \\
	&\phantom{xx}{}- 4\eps\tau\bigg((\bm{n}\times\bm{a})\times\bm{a}
	+ \bm{n}\times\pa_t^0\bm{a})\bigg),
	\nonumber
\end{align}
where $a_0=\tilde{a}_0-V$, $\pa_t^0\bm{a}$ is the lowest-order approximation of 
$\pa_t\bm{a}$ with respect to $\tau$, 
and $\na^\perp=\na\times\bm{e}_3=(\pa_{x_2},-\pa_{x_1},0)^T$.
\end{theorem}

The Lagrange multipliers $a_0$ and $\bm{a}$ are nonlocal functions of the densities
$n_0$ and $\bm{n}$ via the constraint $\langle\M(N)\rangle = N$. System
\eqref{1.n0}--\eqref{1.n} is formally closed but in a very implicit way. 
The proof of the theorem is based on a specification of the quantum Maxwellian 
$\M(N)$ and the transport operator $\T$ in terms of the Pauli basis. Our
arguments are only formal since a rigorous treatment is, even in simple cases,
out of reach. 
In the classical case $\eps=0$,
equation \eqref{1.n0} reduces to the standard drift-diffusion equation
$$
  \pa_t n_0 = \diver(n_0\na a_0 + n_0\na V), \quad\mbox{where }a_0=\log(n_0/(2\pi)).
$$
The terms involving $\alpha$ in \eqref{1.n0}--\eqref{1.n} are coming from the
Rashba interaction. The expression of order $\eps^{-1}$ in the second line of
\eqref{1.n} can be reformulated by using the Grassmann vector identity as
$$
  \bm{a}\langle p^\perp\cdot\M(N)\rangle + \bm{J}^T\bm{a}^\perp
	= \langle p^\perp\times(\bm{a}\times\M(N)\rangle.
$$
This is exactly the corresponding expression in the model of
\cite[Formula (24)]{BHJ21}.

Since the semiclassical expansion of $\bm{a}$ is
$\bm{a}=\bm{n}/n_0+O(\eps^2)$ (Lemma \ref{lem.a}),
we can write \eqref{1.n0}--\eqref{1.n}, up to $O(\eps^2)$, 
as the following cross-diffusion system:
$$
  \pa_t\begin{pmatrix} n_0 \\ \bm{n}\end{pmatrix}
	= \tau\diver\left(\begin{pmatrix} 1-\eps^2|\bm{n}/n_0|^2 & \eps^2\bm{n}^T/n_0 \\
	-\bm{n}/n_0 & \mathbb{I} \end{pmatrix}
	\na\begin{pmatrix} n_0 \\ \bm{n}\end{pmatrix}\right)
	+ f(n_0,\bm{n},\na n_0,\na\bm{n}),
$$
where $\mathbb{I}$ is the identity matrix in $\R^{3\times 3}$ and 
$f$ contains the lower-order terms. The density matrix is positive definite
if $\eps|\bm{n}|<n_0$, and under this condition, the real parts
of the eigenvalues of the diffusion matrix are positive.
This indicates that the nonlocal system is of parabolic type
in the sense of Petrovskii.

Our second main result is a semiclassical expansion, up to second order,
of the nonlocal model \eqref{1.n0}--\eqref{1.n}.

\begin{theorem}[Local quantum-spin model]\label{thm.semi}
Let $N=n_0\sigma_0+\eps\bm{n}\cdot\bm{\sigma}$ be a solution to 
\eqref{1.n0}--\eqref{1.n}. Then $(n_0,\bm{n})$ solves, neglecting terms of order
$O(\alpha^m\eps^n)$ with $m+n>2$,
\begin{align}
  \pa_t n_0 &= \tau\diver\bigg(\na n_0+n_0\na V - \frac{\eps^2}{6}n_0
	\na\frac{\Delta\sqrt{n_0}}{\sqrt{n_0}}\bigg), \label{1.n0loc} \\
	\pa_t\bm{n} &= \tau\diver(\na\bm{n}+\bm{n}\na V) 
	- 2\alpha\tau(2\na^\perp\times\bm{n}+\na^\perp V\times\bm{n})
	- 4\alpha^2\tau(2\bm{n}+\bm{n}^{\perp\perp}) \label{1.nloc} \\
	&\phantom{xx}{}+ \frac{\eps^2}{6}\frac{\bm{n}}{n_0}\times B(N)
	+ \frac{\eps^2\tau}{12}\diver\big(\bm{n}A(N) - \na\Delta\bm{n} + \na\bm{n}C(N)
	+ B(N)\na n_0 + D(N)\big) \nonumber \\
	&\phantom{xx}{}+ \frac{\eps^3\tau}{3}\bm{n}\times\bigg(\frac{\bm{n}}{n_0}
	\times B(N) - B(N)\bigg), \nonumber 
\end{align}
where $\bm{n}^{\perp\perp}=(-n_1,-n_2,0)^T$,
\begin{align*}
  A(N) &= 2\bigg|\frac{\na n_0}{n_0}\bigg|^2\frac{\na n_0}{n_0}
	- 4\frac{\bm{n}}{n_0}\cdot\frac{\na\bm{n}}{n_0}
	- \frac{\na n_0}{n_0}\frac{\Delta n_0}{n_0}
	- \frac{\na n_0}{n_0}\frac{(\na\otimes\na)n_0}{n_0}, \\
	B(N) &= \frac{\Delta\bm{n}}{n_0} - \frac{\na\bm{n}}{n_0}\cdot\frac{\na n_0}{n_0}
	= \diver\bigg(\frac{\na\bm{n}}{n_0}\bigg), \\
	C(N) &= \bigg(\frac{\Delta n_0}{n_0} - \bigg|\frac{\na n_0}{n_0}\bigg|^2
	+ 4\bigg|\frac{\bm{n}}{n_0}\bigg|^2\bigg)\sigma_0
	+ \frac{(\na\otimes\na)n_0}{n_0}, \\
	D(N) &= \frac{\na n_0}{n_0}\na\otimes\na \bm{n}
	- \bigg(\na\bm{n}\cdot\frac{\na n_0}{n_0}\bigg)\frac{\na n_0}{n_0}.
\end{align*}
\end{theorem}

Equation \eqref{1.n0loc} for $n_0$ is decoupled from \eqref{1.nloc}. It 
corresponds to the quantum drift-diffusion or density-gradient model \cite{AnIa89}.
The spin density satisfies, at lowest order in $\alpha$ and $\tau$,
a drift-diffusion equation. In the general case $\eps>0$, equation \eqref{1.nloc} 
is a parabolic equation of fourth order
with $-\Delta^2\bm{n}$ being the highest-order derivative term. The (formal) proof
of Theorem \ref{thm.semi} is based on the semiclassical expansion of the
quantum Maxwellian and the Lagrange multipliers $a_0$ and $\bm{a}$.
Here, the assumption of small polarizations is crucial to be able to compute
the expressions in a suitable way.

\subsection{Comparison with models in the literature}

The local model \eqref{1.n0loc}--\eqref{1.nloc} includes other equations in the
literature as special cases. First, we claim that if $n_1$ and $n_2$ vanish,
then $n_\pm=n_0\pm\eps n_3$ solve, up to order $O(\eps^2)$,
the {\em two-component spinorial quantum drift-diffusion equations}
\begin{equation}\label{1.npm}
  \pa_t n_\pm = \tau\diver(\na n_\pm + n_\pm\na V)
	- \frac{\eps^2}{6}\tau\diver\bigg(n_\pm\na
	\frac{\Delta\sqrt{n_\pm}}{\sqrt{n_\pm}}\bigg)
	- 4\alpha^2\tau(n_\pm - n_\mp).
\end{equation}
Indeed, the third component of \eqref{1.nloc} equals in case $n_1=n_2=0$,
$$
  \pa_t n_3 = \tau\diver(\na n_3+n_3\na V) - 8\alpha^2\tau n_3 + O(\eps^2).
$$
Adding this equation to or subtracting it from \eqref{1.n0loc} gives
$$
  \pa_t(n_0\pm \eps n_3) = \tau\diver\big(\na(n_0\pm \eps n_3)+(n_0\pm \eps n_3)
	\na V\big) - \frac{\eps^2}{6}\diver\bigg(n_0\na\frac{\Delta\sqrt{n_0}}{\sqrt{n_0}}
	\bigg) - 8\alpha^2\eps\tau n_3 + O(\eps^3).
$$
Replacing $n_3=(n_+ - n_-)/(2\eps)$ and expanding
$\sqrt{n_0} = \sqrt{n_\pm}\sqrt{1 \mp \eps n_3/n_\pm} = \sqrt{n_\pm} + O(\eps)$,
then gives \eqref{1.npm} up to order $O(\eps^2)$. 
Equation \eqref{1.npm} corresponds to the two-component
drift-diffusion model for the spin-up and spin-down densities $n_+$ and
$n_-$, respectively, which was derived in \cite[Theorem~2]{BaMe10} from
the Wigner--BGK model.
The expression $J_\pm=\na n_\pm + n_\pm\na V$ is the classical
contribution of the spin-up/spin-down current density. 
The second term on the right-hand side of
\eqref{1.npm} can be interpreted as a quantum current including the Bohm
potential $\Delta\sqrt{n_\pm}/\sqrt{n_\pm}$. The equations are weakly coupled
through the last term, which expresses the well-known D'yakonov--Perel' spin 
relaxation. The spin drift-diffusion model with $\eps=0$ was suggested in
\cite{ZFD02} and mathematically analyzed in \cite{Gli08,GlGa10}.
Model \eqref{1.npm} in one space dimension and with nonlinear diffusion 
corresponds to the bipolar quantum drift-diffusion equations
that were analyzed in \cite{ChCh09}.

Second, 
we observe that equation \eqref{1.n0loc} for $n_0$ is decoupled from \eqref{1.nloc}
since it does not contain the spin density $\bm{n}$. In fact, both equations
are completely decoupled in the limit $\eps\to 0$. Indeed, in this limit,
equations \eqref{1.n0loc}--\eqref{1.nloc} become the {\em spin-vector drift-diffusion
model}
\begin{align}
  \pa_t n_0 &= \tau\diver(\na n_0+n_0\na V), \label{1.n0elhajj} \\
  \pa_t\bm{n} &= \tau\diver(\na\bm{n}+\bm{n}\na V) - 2\alpha\tau(2\na^\perp\times\bm{n}
	+ \na^\perp V\times\bm{n}) - 4\alpha^2\tau(2\bm{n}+\bm{n}^{\perp\perp}).
	\label{1.nelhajj}
\end{align}
These equations correspond to the model of \cite[Section 4.3]{BHJ21}
and to the semiclassical drift-diffusion equations derived in \cite{ElH14}
in the case of constant relaxation time and purely spin-orbit interaction field.
The charge density $n_0$ satisfies the standard drift-diffusion equation
for semiconductors. Since
$$
  \na^\perp\times\bm{n}
	= \pa_{x_1}\begin{pmatrix} n_3 \\ 0 \\-n_1 \end{pmatrix}
	+ \pa_{x_2}\begin{pmatrix} 0 \\ n_3 \\ -n_2 \end{pmatrix}
  = \diver\begin{pmatrix} n_3 & 0 \\ 0 & n_3 \\ -n_1 & -n_2 \end{pmatrix},
$$
the spin current diffuses according to the classical drift-diffusion current
and an additional current, coming from the spin-orbit interaction. 
The equation for $\bm{n}$ also contains the gate control term 
$-2\alpha\tau\na^\perp V\times\bm{n}$, which expresses the capability to
control the spin by means of an applied voltage, and the relaxation term
$- 4\alpha^2\tau(2\bm{n}+\bm{n}^{\perp\perp})=-4\alpha^2\tau(n_1,n_2,2n_3)^T$.

A related spin-vector model was derived in \cite{ZaJu13}, leading to \eqref{1.n0elhajj}
and an equation similar to \eqref{1.nelhajj}. The difference to the model
of \cite{ZaJu13} is that there, quantum effects are taken into account but only
up to first order. Indeed, it is assumed in \cite{ZaJu13} that $\eps$ and $\tau$
are of the same order such that second-order effects, like the quantum Bohm
potential, cannot be seen in this approach.

Third, the nonlocal model \eqref{1.n0}--\eqref{1.n} reduces in the spinless case
to the following nonlocal equation for the charge density:
$$
  \pa_t n_0 = \tau\diver(n_0\na(a_0+V)), \quad
	n_0 = \langle\M(N)\rangle 
	= \bigg\langle\mathcal{W}\bigg(\exp\mathcal{W}^{-1}\bigg(-\frac{|p|^2}{2}+a_0
	\bigg)\bigg)\bigg\rangle,
$$
which was derived in \cite{DGM07} as the {\em entropic quantum drift-diffusion model}.
An interesting feature of this model is that the macroscopic quantum free energy
$$
  E = -\int_{\R^2}n_0(a_0+V)dx
$$
is a decreasing function of time:
\begin{align*}
  \frac{dE}{dt} &= \int_{\R^2}\big(\pa_t n_0(a_0+V) + n_0\pa_t(a_0+V)\big)dx
	= \int_{\R^2}\pa_t n_0(a_0+V+1)dx \\
	&= -\int_{\R^2}n_0|\na(a_0+V)|^2 dx \le 0.
\end{align*}
Here, we have used the property that the derivative
$\pa_t n_0$ equals $n_0\pa_t a_0$ (this is basically a consequence of
\cite[Lemma 3.3]{DeRi03}).

\subsection*{Notation}

We summarize some notation used in this paper. Bold face letters indicate
vectors in $\R^3$ like $\bm{a}=(a_1,a_2,a_3)^T\in\R^3$. We write
$\langle f\rangle=\int_{\R^2}fdp$ and introduce the notation
$$
  p^\perp=(p_2,-p_1,0)^T, \quad\bm{n}^{\perp\perp}=(-n_1,-n_2,0)^T, \quad
  \na^\perp_x=(\pa_{x_2},-\pa_{x_1},0)^T.
$$
If clear from the context, we write $\na$ instead of $\na_x$. The
partial derivative with respect to $x_i$ or $p_k$ is denoted by
$\pa_{x_i}$ or $\pa_{p_k}$, respectively, and $\pa^2_{x_i x_k}$ is
a second-order partial derivative.


The paper is organized as follows. 
In Section \ref{sec.prep}, we present some background material, in particular
the von Neumann and Wigner equations, the Moyal product and its properties, 
and introduce the quantum Maxwellian $\M(N)$ and the Wigner--BGK model \eqref{1.swe}. 
The nonlocal quantum model
of Theorem \ref{thm.full} is derived in Section \ref{sec.full}, while the
semiclassical expansion leading to the local quantum model of Theorem \ref{thm.semi}
is performed in Section \ref{sec.semi}. The appendices collect some technical
proofs, namely the formal solution of the quantum maximum entropy problem
leading to the quantum Maxwellian and its semiclassical expansion.


\section{Background material}\label{sec.prep}

In this section, we detail the scaling of the von Neumann equation,
introduce the phase-space formulation, 
and define the Moyal product and the quantum Maxwellian.

\subsection{Scaling}\label{sec.scal}

The confined electrons move in the plane $x=(x_1,x_2,0)^T$ with the momentum
$p=(p_1,p_2,0)^T$. The electron spin, however, is a vector in $\R^3$ having
generally nonvanishing components. An electron in the $(x_1,x_2)$-plane with
Rashba interaction is described by the Hamiltonian
$$
  \mathcal{H} = \bigg(-\frac{\hbar^2}{2m}\Delta + qV(x)\bigg)\sigma_0
	- \hbar\alpha_R\begin{pmatrix} 0 & \ii\pa_{x_2}-\pa_{x_1} \\
	\ii\pa_{x_2}+\pa_{x_1} & 0 \end{pmatrix},
$$
where the parameters are the reduced Planck constant $\hbar$, the electron mass $m$,
the elementary charge $q$, and the Rashba constant $\alpha_R>0$. 
Furthermore, $V(x)$ is the given electric 
potential and $\sigma_0\in\C^{2\times 2}$ is the identity matrix. 
The evolution of the electrons is governed by
the von Neumann equation for the density operator $\hat\rho(t)$, 
which is a positive trace-class operator on $L^2(\R^2;\C^2)$,
$$
  \ii\hbar\pa_t\hat\rho = [\mathcal{H},\hat\rho], \quad t>0.
$$

This equation can be written in dimensionless form by introducing the
reference length $x_0$ (e.g.\ the device diameter), 
time $t_0$, and density $N_0$. We choose the
thermal momentum $p_0=\sqrt{mk_BT_0}$ (where $k_B$ is the Boltzmann constant
and $T_0$ the background temperature), the reference potential $V_0=p_0^2/(mq)$,
and the energy time $t_E=mx_0/p_0$. The energy time corresponds to the time
that a typical electron with energy $k_BT_0$ needs to cross the device.
The time $t_0$ denotes another time scale and will be discussed in Section \ref{sec.QM}.
Then, using the same notation for the
unscaled and scaled variables, the scaled von Neumann equation becomes
\eqref{1.vne}, the scaled Hamiltonian equals \eqref{1.H},
and the scaled energy time, Planck constant, and Rashba constant are given by,
respectively,
$$
  \tau_0 = \frac{t_E}{t_0}, \quad \eps = \frac{\hbar}{x_0p_0}, \quad
	\alpha = \frac{mx_0\alpha_R}{\hbar}.
$$

\subsection{Phase-space formulation}\label{sec.phase}

For the asymptotic analysis, it is convenient to
work with phase-space functions instead of density operators. We use the
Wigner transformation to transform a density operator into the phase-space-type Wigner
function. Of course, due to Heisenberg's uncertainty principle, it is impossible
to have have a phase-space description in quantum mechanics. The 
Wigner function is formally similar to a phase-space distribution, and the
Wigner transformation can be considered as a tool to simplify the calculations
and to obtain a classical-like physical intuition behind the mathematical
manipulations. For details, we refer to \cite{Jue09,LoMo13}. 

The density operator $\hat\rho$ in \eqref{1.vne} is a 
(time-dependent) Hilbert--Schmidt
operator on the space $L^2(\R^2;\C^2)$ \cite[Chap.~6]{ReSi72}. It is uniquely
determined by its kernel $\rho\in L^2(\R^2\times\R^2;\C^{2\times 2})$
satisfying
$$
  (\hat\rho\psi)(x) = \int_{\R^2}\rho(x,y)\psi(y)dy \quad\mbox{for }
	\psi\in L^2(\R^2;\C^2).
$$
The Wigner transform $\mathcal{W}(\rho)$ is a matrix-valued function
of the phase-space variables $(x,p)$, defined by
\begin{equation}\label{2.wigner}
  \mathcal{W}(\hat\rho)(x,p) = \int_{\R^2}
	\rho\bigg(x+\frac{\eta}{2},x-\frac{\eta}{2}\bigg)e^{-\ii\eta\cdot p/\eps}
	d\eta.
\end{equation}
Note that the integration domain is $\R^2$ and not $\R^3$, since the
electron system is confined in the two-dimensional plane with respect to $x$ and $p$.
The Wigner transform defined for Hilbert--Schmidt operators can be extended
to a wider class of distributional phase-space functions \cite{Fol89}. 
In such an extended setting, the Wigner transformation is the inverse
of the Weyl quantization, which assigns to a phase-space function (or distribution)
a quantum operator and which is defined for suitable Wigner functions $W$ by
$$
  \mathcal{W}^{-1}(W)(x,y) = \frac{1}{(2\pi\eps)^2}\int_{\R^2}
	W\bigg(\frac{x+y}{2},p\bigg)e^{\ii (x-y)\cdot p/\eps}dp.
$$
In the literature, often the expression {\em symbols} is
used for phase-space functions (or distributions) associated to operators via
Wigner--Weyl transforms, while the expression {\em Wigner function} is reserved
to those symbols that are the Wigner transforms of density operators. 

Let $W=\mathcal{W}(\rho)$ be a symbol associated to a density operator
(i.e.\ a Wigner function). We can express $W$ in the terms of the Pauli basis,
$$
  W(x,p) = \sum_{j=0}^3w_j(x,p)\sigma_j 
	=: w_0(x,p)\sigma_0 + \bm{w}(x,p)\cdot\bm{\sigma},
$$
where the Pauli matrices
$$
  \sigma_0 = \begin{pmatrix} 1 & 0 \\ 0 & 1 \end{pmatrix}, \quad
  \sigma_1 = \begin{pmatrix} 0 & 1 \\ 1 & 0 \end{pmatrix}, \quad
	\sigma_2 = \begin{pmatrix} 0 & -\ii \\ \ii & 0 \end{pmatrix}, \quad
	\sigma_3 = \begin{pmatrix} 1 & 0 \\ 0 & -1 \end{pmatrix}
$$
are a basis of Hermitian matrices of $\C^{2\times 2}$. 
For instance, the Wigner transform of the Hamiltonian \eqref{1.H} equals
\begin{equation}\label{2.Heps}
\begin{aligned}
  & H_\eps = \mathcal{W}(\mathcal{H}_\eps)(x,p) 
	= \eta_0\sigma_0 + \bm{\eta}\cdot\bm{\sigma}, \\
	& \mbox{where }\eta_0(x,p):=\tfrac12|p|^2+V(x),\quad 
	\bm{\eta}(x,p):=\alpha p^\perp, 
\end{aligned}
\end{equation}
and $p^\perp:=p\times\bm{e}_3=(p_2,-p_1,0)^T$, $\bm{e}_3=(0,0,1)^T$.

The Pauli algebra is quite convenient for mathematical manipulations.
For instance, we note the following rule. Let 
$A=a_0\sigma_0+\bm{a}\cdot\bm{\sigma}$ and $B=b_0\sigma_0+\bm{b}\cdot\bm{\sigma}$
be two matrices in $\C^{2\times 2}$. Then
\begin{equation}\label{2.pauli}
  AB = (a_0b_0+\bm{a}\cdot\bm{b})\sigma_0 + (a_0\bm{b}+b_0\bm{a}+\ii\bm{a}\times\bm{b})
	\cdot\bm{\sigma}, \quad \operatorname{tr}(AB) = 2(a_0b_0+\bm{a}\cdot\bm{b}).
\end{equation}

\subsection{Moyal product}

The Moyal product appears when transforming the von Neumann equation \eqref{1.vne}
to the Wigner equation. In fact, the concatenation of operators translates
into the Moyal product of the Wigner transforms. We refer to \cite{Fol89}
for proofs of the results mentioned in this section. For two symbols
$f$, $g\in L^2(\R^2\times\R^2;\C)$, the Moyal product is defined as the generalized
convolution
\begin{align}
  (f\#g)(x,p) &= \frac{1}{(\pi\eps)^4}\int_{\R^2}\int_{\R^2}\int_{\R^2}\int_{\R^2}
	f(x_1,p_1)g(x_2,p_2) \label{def.moyal} \\
	&\phantom{xx}{}\times
	\exp\bigg(\frac{2\ii}{\eps}\big((x-x_2)\cdot p_1 + (x_1-x)\cdot p_2 
	- (x_1-x_2)\cdot p\big)\bigg)dx_1dp_1dx_2dp_2. \nonumber
\end{align}

\begin{lemma}\label{lem.hash}
Let $\hat\rho_1$, $\hat\rho_2$ be two density operators on $L^2(\R^2;\C^2)$. Then
$$
  \mathcal{W}(\hat\rho_1\hat\rho_2) = \mathcal{W}(\hat\rho_1)\#
	\mathcal{W}(\hat\rho_2).
$$
Furthermore, if $f$, $g$ are two symbols with values in $\C$ then
\begin{equation}\label{moyal.int}
  \int_{\R^2}\int_{\R^2}(f\# g)(x,p)dxdp 
	= \int_{\R^2}\int_{\R^2}f(x,p)g(x,p)dxdp.
\end{equation}
\end{lemma}

The lemma is formally proved by straightforward calculations 
using the Weyl quantization.

For the next result, we need a multi-index notation. Let $\mu=(\mu_1,\mu_2)
\in\N_0^2$ be a multiindex with order $|\mu|=\mu_1+\mu_2$ and factorial
$\mu!=\mu_1!\mu_2!$ and let the partial derivative $\pa_x^\mu$ be
an abbreviation for $\pa^{|\mu|}/(\pa_{x_1}^{\mu_1}\pa_{x_2}^{\mu_2})$
and similarly for $\pa_p^\mu$.
The Moyal product has the following semiclassical expansion.

\begin{lemma}\label{lem.moyaleps}
Let $f$, $g$ be two symbols. Then
\begin{align*}
  & (f\# g)(x,p) = \sum_{j=0}^\infty\eps^j (f\#_{j}g)(x,p), \quad\mbox{where} \\
	& (f\#_{j}h)(x,p) = \frac{1}{(2\ii)^j}\sum_{|\mu|+|\nu|=j}
	\frac{(-1)^{|\mu|}}{\mu!\nu!}\pa_x^\mu\pa_p^\nu f(x,p)
	\pa_p^\mu\pa_x^\nu g(x,p).
\end{align*}
\end{lemma}

The first two terms in the sum are the normal multiplication and the Poisson
bracket, respectively:
\begin{equation}\label{moyal.two}
  f\#_{0}g = fg, \quad
	f\#_{1}g = \frac{1}{2\ii}(\na_p f\cdot\na_x g - \na_x f\cdot\na_p g).
\end{equation}

If $A=(A_{ij})$, $B=(B_{ij})$ are two matrix-valued symbols with values in
$\C^{2\times 2}$, we define its Moyal product as
$(A\# B)_{ij} = \sum_{k=1}^2A_{ik}\# B_{kj}$.
Formulating $A=a_0\sigma_0+\bm{a}\cdot\bm{\sigma}$ and 
$B=b_0\sigma_0+\bm{b}\cdot\bm{\sigma}$ in the Pauli components, the matrix
Moyal product can be written in the Pauli basis as
\begin{equation}\label{moyal.AB}
  A\# B = (a_0\# b_0 + \bm{a}\cdot_\#\bm{b})\sigma_0
	+ (a_0\#\bm{b} + \bm{a}\# b_0 + \ii\bm{a}\times_\#\bm{b})\cdot\bm{\sigma},
\end{equation}
where ``$\cdot_\#$'' and ``$\times_\#$'' are the inner and cross products
on $\R^3$, respectively, where the multiplication is replaced by the Moyal product.

Given two symbols $f$ and $g$, we define the odd and even Moyal product by
\begin{equation}\label{moyal.odd}
  f\#_{\rm odd}g = \frac12(f\# g-g\# f), \quad
	f\#_{\rm even}g = \frac12(f\# g+g\# f).
\end{equation}
Let $V=V(x)$ and $f=f(x,p)$ be two symbols. We define the potential operator
\begin{align*}
  & (\theta_\eps[V]f)(x,p) = \frac{1}{(2\eps)^2}\int_{\R^2}\int_{\R^2}
	\delta_\eps[V](x,\eta)f(x,p')e^{-\ii(p-p')\cdot\eta}d\eta dp', \\
	& \mbox{where}\quad \delta_\eps[V] = \frac{1}{\ii\eps}\bigg(
	V\bigg(x+\frac{\eps}{2}\eta\bigg) - V\bigg(x-\frac{\eps}{2}\eta\bigg)\bigg).
\end{align*}

\begin{lemma}\label{lem.moyal}
Let $V=V(x)$ and $f=f(x,p)$ be two symbols. Then
\begin{align}
  & \ii\eps\theta_\eps[V]f = 2V\#_{\rm even}f, \label{2.theta} \\
	& \langle\theta_\eps[V](f)\rangle = 0, \quad
	\langle p\theta_\eps[V](f)\rangle = -\na_x V(f). \label{2.theta2} 
\end{align}
\end{lemma}

\begin{proof}
Using the definition of the Moyal product, it follows after suitable substitutions
that
\begin{align*}
  2(V\#_{\rm even}f)(x,p) &= (V\# f - f\# V)(x,p) \\
	&= \frac{1}{(2\pi)^2}\int_{\R^2}\int_{\R^2}\bigg(V\bigg(x+\frac{\eps}{2}\eta\bigg)
	- V\bigg(x-\frac{\eps}{2}\eta\bigg)\bigg)f(x,p')e^{\ii\eta\cdot(p'-p)}d\eta dp' \\
	&= \ii\eps(\theta_\eps[V]f)(x,p). 
\end{align*}
A formal proof of \eqref{2.theta2} can be found in \cite[Lemma 12.9]{Jue09}.
\end{proof}

By Lemma \ref{lem.moyaleps}, the operator $\theta_\eps[V]$ can be expanded as
$$
  \theta_\eps[V]f = \na_x V\cdot\na_p f + O(\eps^2),
$$
which shows that it reduces in the limit $\eps\to 0$ to the classical drift term 
appearing in kinetic theory.

Let $\hat\rho$ be a density operator on $L^2(\R^2;\C^2)$ with
Wigner function $W=\mathcal{W}(\hat\rho)$. Then
\begin{equation*}
  \operatorname{Tr}(\hat\rho) = \frac{1}{(2\pi\eps)^2}\operatorname{tr}
	\int_{\R^2}\int_{\R^2}W(x,p)dxdp,
\end{equation*}
where ``Tr'' is the operator trace and ``tr'' the matrix trace.
Furthermore, let $\hat\rho_1$ and $\hat\rho_2$ be two density
operators and let
$W_1=\mathcal{W}(\hat\rho_1)$, $W_2=\mathcal{W}(\hat\rho_2)$ be the associated
Wigner functions. Then it follows from \eqref{moyal.int} and
$\mathcal{W}(\hat\rho_1\hat\rho_2)=W_1\# W_2$ that
\begin{equation}\label{2.tr}
  \operatorname{Tr}(\hat\rho_1\hat\rho_2)
	= \frac{1}{(2\pi\eps)^2}\operatorname{tr}\int_{\R^2}\int_{\R^2}W_1(x,p)W_2(x,p)dxdp.
\end{equation}

The Moyal product allows us to formulate the von Neumann equation in the
phase-space setting. 

\begin{lemma}\label{lem.w}
Let $\hat\rho$ be a solution to the von Neumann equation \eqref{1.vne}
and $W=\mathcal{W}(\hat\rho)$ be its Wigner function. Then $W$ solves
$$
  \tau_0\pa_t W + \T W = 0, \quad t>0,
$$
where $\T W:=(\ii/\eps)[H_\eps,W]_\#=(\ii/\eps)(H_\eps\# W - W\# H_\eps)$. 
Furthermore, the Pauli components
of $W=w_0\sigma_0+\bm{w}\cdot\bm{\sigma}$ solve
\begin{align}
  \tau_0\pa_t w_0 + p\cdot\na_x w_0 + \alpha\eps\na_x^\perp\cdot\bm{w} 
	- \theta_\eps[V]w_0 &= 0, \label{2.w0} \\
	\tau_0\pa_t\bm{w} + p\cdot\na_x\bm{w} + \alpha\eps\na_x^\perp w_0
	- \theta_\eps[V]\bm{w} - 2\alpha p^\perp\times\bm{w} &= 0, \quad t>0, \label{2.w}
\end{align}
recalling that $\na_x^\perp=(\pa_{x_2},-\pa_{x_1},0)^T$ and
$p^\perp=(p_2,-p_1,0)^T$.
\end{lemma}

The lemma shows that the transport operator can be written as 
\begin{align}
  \T W &= \big(p\cdot\na_x + \alpha\eps\na_x^\perp\cdot\bm{w} 
	- \theta_\eps[V]w_0\big)\sigma_0 \label{2.TW} \\
	&\phantom{xx}{}+ \big(p\cdot\na_x\bm{w} + \alpha\eps\na_x^\perp w_0
	- \theta_\eps[V]\bm{w} - 2\alpha p^\perp\times\bm{w}\big)\cdot\bm{\sigma}.
	\nonumber
\end{align}

\begin{proof}[Proof of Lemma \ref{lem.w}]
Applying the Wigner transform to \eqref{1.vne}, 
\begin{equation}\label{2.HW}
  \ii\eps\tau_0\pa_t W = \mathcal{W}(\ii\eps\tau_0\pa_t\hat\rho) 
	= \mathcal{W}(\mathcal{H}_\eps\hat\rho)
	- \mathcal{W}(\hat\rho\mathcal{H}_\eps) = H_\eps\# W - W\# H_\eps, 
\end{equation}
where $H_\eps$ is given by \eqref{2.Heps}.
Then \eqref{moyal.AB}, \eqref{moyal.odd}, and an elementary computation show that
$$
  H_\eps\# W - W\# H_\eps = 2\big(\eta_0\#_{\rm odd}w_0 + \bm{\eta}\cdot_{\#_{\rm odd}}
	\bm{w}\big)\sigma_0 + 2\big(\eta_0\#_{\rm odd}\bm{w} + \bm{\eta}\#_{\rm odd}w_0
	+ \ii\bm{\eta}\times_{\#_{\rm even}}\bm{w}\big)\cdot\bm{\sigma},
$$
where $\eta_0$ and $\bm{\eta}$ are defined in \eqref{2.Heps}.
Comparing the Pauli components of the left-hand side of \eqref{2.HW}, written as
$\ii\eps\tau\pa_t(w_0\sigma_0 + \bm{w}\cdot\bm{\sigma})$, with those from
the right-hand side, we find that
\begin{align*}
  \ii\eps\tau_0\pa_t w_0 &= \eta_0\#_{\rm odd}w_0 
	+ \bm{\eta}\cdot_{\#_{\rm odd}}\bm{w}, \\
	\ii\eps\tau_0\pa_t\bm{w} &= \eta_0\#_{\rm odd}\bm{w} + \bm{\eta}\#_{\rm odd}w_0
	+ \ii\bm{\eta}\times_{\#_{\rm even}}\bm{w}.
\end{align*}
It remains to evaluate the right-hand sides. 
It follows from \eqref{2.theta} that $2V\#_{\rm odd}w_0=\ii\eps\theta_\eps[V]w_0$.
Furthermore, since the derivatives of $|p|^2/2$ of order higher than two vanish,
the Moyal product $|p|^2\#_{\rm odd}w_0$ reduces to $|p|^2\#_1 w_0
= -\ii\eps p\cdot\na_x w_0$ (see \eqref{moyal.two}). Hence,
$$
  \eta_0\# w_0 = \ii\eps(\theta_\eps[V]w_0 - p\cdot\na_x w_0).
$$
The higher-order derivatives of $\bm{\eta}$ vanish too such that
$$
  2\bm{\eta}\cdot_{\#_{\rm odd}}\bm{w} = -\ii\alpha\eps\sum_{j=1}^2
	\na_p \eta_j\cdot\na_x w_j = -\ii\alpha\eps(\pa_{x_2}w_1-\pa_{x_1}w_2)
	= -\ii\alpha\eps\na_x^\perp\cdot\bm{w}.
$$
Collecting the last two displayed expressions, we obtain \eqref{2.w0}. 

Similarly as above, we have
$$
  2\eta_0\#\bm{w} = \ii\eps(\theta[V]\bm{w}-p\cdot\na_x\bm{w}), \quad
	2\bm{\eta}\cdot_{\#_{\rm odd}}w_0 = -\ii\alpha\eps\na_x^\perp w_0,
$$
where $p\cdot\na_x\bm{w}=\sum_{j=1}^2 p_j\pa_{x_j}\bm{w}$. 
Again, since the higher-order
derivatives of $\bm{\eta}$ vanish, only the lowest-order term
of the even Moyal cross product $\bm{\eta}\times_{\#_{\rm even}}\bm{w}$ remains:
$$
  \ii\bm{\eta}\times_{\#_{\rm even}}\bm{w} 
	= \ii\bm{\eta}\times_{\#_{0}}\bm{w} = \ii\bm{\eta}\times\bm{w}
	= \ii\alpha p^\perp\times\bm{w}.
$$
We deduce \eqref{2.w} from the last three displayed expressions, finishing the proof.
\end{proof}


\subsection{Quantum Maxwellian and Wigner--Boltzmann equation}\label{sec.QM}

The local equilibrium state of the electron gas is assumed to be the minimizer
of the quantum entropy functional (if it exists) under the constraints of
given macroscopic densities \cite{DeRi03}. 
The quantum maximum entropy problem means that the
collisions drive the system towards the most probable state compatible with
the observed densities. The entropy functional is the quantum free energy
$$
  \mathcal{G}(\hat\rho) = \operatorname{Tr}(\hat\rho\log\hat\rho - \hat\rho
	+ \mathcal{H}_\eps\hat\rho),
$$
where Tr is the operator trace, log is the operator logarithm, 
and $\mathcal{H}_\eps$ is the Hamiltonian \eqref{1.H}. Note that the
operator logarithm is well defined for positive definite density operators.
To formulate the entropy functional in the phase space, we introduce the
quantum exponential and quantum logarithm according to \cite{DMR05} by
$$
  \Exp(W) := \mathcal{W}(\exp \mathcal{W}^{-1}(W)), \quad
	\Log(W) := \mathcal{W}(\log \mathcal{W}^{-1}(W)),
$$
where exp is the exponential operator. We deduce from identity \eqref{2.tr} that
\begin{equation}\label{2.E}
  \mathcal{E}(W) := \mathcal{G}(\hat\rho)
	= \frac{1}{(2\pi\eps^2}\operatorname{tr}\int_{\R^2}\int_{\R^2}
	(W\Log W-W+H_\eps W)dxdp,
\end{equation}
where $H_\eps$ is defined in \eqref{2.Heps} and $W=\mathcal{W}(\hat\rho)$. 

\begin{definition}[Quantum maximum entropy problem]
Given the numbers $n_0$ and $\bm{n}=(n_1,n_2,n_3)^T\in\R^3$ satisfying
$\eps|\bm{n}|<n_0$, 
we wish to find the Wigner function $W^*$ such that $\mathcal{E}(W^*)$ is minimal among
all symbols $W=w_0\sigma_0+\bm{w}\cdot\bm{\sigma}$ such that $\mathcal{W}^{-1}(W)$
is positive definite and 
$$
  \langle w_0\rangle = n_0, \quad \langle\bm{w}\rangle = \eps\bm{n}.
$$
\end{definition}

Observe that we introduced a smallness condition on the spin components.
We suppose that the spin vector is of the order of the scaled Planck constant.
This condition simplifies the semiclassical expansion, and it
implies that the density matrix $N:=\langle W\rangle 
= n_0\sigma_0+\eps\bm{n}\cdot\bm{\sigma}$ is positive definite.
The positive definiteness condition on $\mathcal{W}^{-1}(W)$ guarantees that
the quantum logarithm is well defined. 

\begin{theorem}\label{thm.QM}
If the quantum maximum entropy problem has a solution $\M(N)=\M_0\sigma_0
+ \bm{\M}\cdot\bm{\sigma}$, then it is necessarily of the form
$$
  (\M(N))(x,p) = \Exp\big(-H_\eps(x,p) + \tilde{a}_0(x)\sigma_0
	+ \eps\bm{a}(x)\cdot\bm{\sigma}\big),
$$
where $\tilde{a}_0$ and $\bm{a}$
are real Lagrange multipliers. The solution satisfies the constraints
$$
  n_0 = \frac12\langle\operatorname{tr}(\M(N)\sigma_0)\rangle 
	= \langle\M_0\rangle, \quad
	\eps n_j = \frac12\langle\operatorname{tr}(\M(N)\sigma_j)\rangle, \quad j=1,2,3.
$$
\end{theorem}

We call $\M=\M(N)$ the {\em quantum Maxwellian}. It extends slightly the notion
of the quantum Maxwellian introduced in \cite{DeRi03}. The proof of the
existence of the quantum Maxwellian is a very difficult task, 
even in the one-dimensional case \cite{DuMe19,MePi10}. Regularity properties
of $\M$ are proved in \cite{MePi17}. 
The proof of Theorem \ref{thm.QM} is deferred to Appendix \ref{app.QM}.

We define the Hermitian matrix of Lagrange multipliers by
$\tilde{A}(x) = \tilde{a}_0\sigma_0 + \eps\tilde{\bm{a}}\cdot\bm{\sigma}$.
Then (see \eqref{2.Heps})
\begin{equation}\label{2.HepsA}
  -H_\eps+\tilde{A} = h_0\sigma_0 + \eps\bm{h}_1\cdot\bm{\sigma}, 
	\quad h_0 = -\tfrac12|p|^2 + a_0, \ a_0 = \tilde{a}_0-V,\ 
	\bm{h}_1 = \bm{a} - \alpha p^\perp.
\end{equation}

The definition of the quantum Maxwellian allows us to introduce the
relaxation-time (BGK-type) collision operator
$Q(W)=\tau^{-1}(\M(N)-W)$ into the transport model, where $\tau>0$ is a 
scaled relaxation time, leading to 
$$
  \tau_0\pa_t W + \T W = Q(W),
$$
where the scaled time $\tau_0$ is introduced in Section \ref{sec.scal}
and we recall the definition $\T W = (\ii/\eps)(H_\eps\# W - W\# H_\eps)$
(see Lemma \ref{lem.w}).
The collision operator conserves the particle number and spin since, by
definition of the quantum Maxwellian, $\langle Q(W)\rangle = 0$.

We assume that $\tau_0$ is of order one and $\tau$ is small compared to one.
Physically this means that the time scale of the system is the energy time
$t_E=mx_0/p_0$ and the relaxation time is small compared to $t_E$. This
leads to the Wigner--Boltzmann equation in the hydrodynamic scaling
\begin{equation}\label{2.wbe}
  \tau\pa_t W + \tau\T W = \M(N)-W, \quad t>0, \quad\mbox{where }N = \langle W\rangle.
\end{equation}
The existence of solutions to the von Neumann--BGK equation associated to \eqref{2.wbe}
with values in the Schatten space of order one is proved in \cite{MePi17}. 

We already mentioned in the introduction that we cannot use a classical
diffusion scaling (i.e.\ $\tau_0$ and $\tau$ are of the same order and small)
since the moment $\langle\T\M(N)\rangle$ generally does not vanish.
The following proposition makes this statement more precise.
We recall the notation $[A,B]_\#=A\# B-B\# A$ for two symbols $A$ and $B$.

\begin{lemma}\label{lem.comm}
Let $W$ be a solution to \eqref{2.wbe} and $\tilde{A}$ 
be the Lagrange multiplier matrix
related to the Maxwellian by $\M(N)=\Exp(-H_\eps+\tilde{A})$. 
Then $\langle \T\M(N)\rangle=0$ if and only if $\langle[\tilde{A},\M(N)]_\#\rangle=0$. 
In particular, $\langle \T\M(N)\rangle=0$ if and only if 
$\tilde{A}$ commutes with $N$.
\end{lemma}

\begin{proof}
We know from Lemma \ref{lem.w} that $\ii\eps\T W=-[H_\eps,W]_\#$.
Moreover, since every operator commutes with its exponential, we have
$(-H_\eps+\tilde{A})\#\M(N)=\M(N)\#(-H_\eps+\tilde{A})=0$. This gives
\begin{equation}\label{2.AM}
  -\ii\eps\T\M(N) = [H_\eps,\M(N)]_\#
	= -[-H_\eps+\tilde{A},\M(N)]_\# + [\tilde{A},\M(N)]_\# 
	= [\tilde{A},\M(N)]_\#,
\end{equation}
showing the first statement. By identity \eqref{moyal.int},
$\langle[A,B]_\#\rangle=\langle A\# B\rangle - \langle B\# A\rangle
= \langle AB\rangle - \langle BA\rangle = \langle[A,B]\rangle$ for any
symbols $A$ and $B$. Therefore, since $\tilde{A}$ only depends on $x$,
\begin{equation*}
  -\ii\eps\langle\T\M(N)\rangle = \langle[\tilde{A},\M(N)]_\#\rangle
	= \langle[\tilde{A},\M(N)]\rangle = [\tilde{A},\langle \M(N)\rangle] 
	= [\tilde{A},N].
\end{equation*}
This proves the second statement.
\end{proof}


\section{Derivation of the nonlocal quantum model}\label{sec.full}

We insert the function $G:=-\tau^{-1}(\M(N)-W)$ into the Wigner--Boltzmann equation 
\eqref{2.wbe} and use the property $W=\M(N)+O(\tau)$:
$$
  G = -\pa_t W - \T W = -\pa_t \M(N) - \T \M(N) + O(\tau).
$$
After integrating \eqref{2.wbe} with respect to $p$ and taking into account
that $\langle W\rangle = N$ and $\langle \M(N)-W\rangle = 0$, we find that
\begin{align*}
  \pa_t N &= -\langle \T W\rangle = -\langle\T \M(N)\rangle - \tau\langle\T G\rangle \\
	&= -\langle\T \M(N)\rangle + \tau\langle\T \pa_t \M(N)\rangle
	+ \tau\langle\T\T \M(N)\rangle + O(\tau^2).
\end{align*}

We wish to compute the terms on the right-hand side. To simplify the notation,
we set $\M:=\M(N)$. 
The proof of Lemma \ref{lem.comm} shows that
$\ii\eps\langle\T\M\rangle = -[\tilde{A},N]$. 
Inserting the Pauli decompositions
$\tilde{A}=\tilde{a}_0\sigma_0 + \eps\bm{a}\cdot\bm{\sigma}$ and 
$N=n_0\sigma_0+\eps\bm{n}\cdot\bm{\sigma}$ and using \eqref{2.pauli}, 
a computation leads to
\begin{equation}\label{3.TM}
  \langle\T\M\rangle = 2\eps(\bm{n}\times\bm{a})\cdot\bm{\sigma}.
\end{equation}

To calculate $\langle\T\T\M\rangle$, we use \eqref{2.AM},
the decomposition $\M=\M_0\sigma_0 + \bm{\M}\cdot\bm{\sigma}$, 
rule \eqref{moyal.AB}, and property \eqref{2.theta}:
\begin{align}
  \T\M &= \frac{1}{\ii\eps}(\M\#\tilde{A} - \tilde{A}\#\M)
	= \frac{2}{\ii\eps}\M\#_{\rm odd}\tilde{A} \nonumber \\
	&= \frac{2}{\ii\eps}\bigg(\M_0\#_{\rm odd}\tilde{a}_0 
	+ \eps\sum_{j=1}^3\M_j\#_{\rm odd}
	a_j\bigg)\sigma_0 + \frac{2}{\ii\eps}\bigg(\eps\M_0\#_{\rm odd}
	\bm{a} + \bm{\M}\#_{\rm odd} \tilde{a}_0 
	+ \ii\eps\M\times_{\#_{\rm even}}\bm{a}\bigg)\cdot
	\bm{\sigma} \nonumber \\
	&= -\bigg(\theta_\eps[\tilde{a}_0](\M_0) 
	+ \eps\sum_{j=1}^3\theta_\eps[a_j](\M_j)\bigg)
	\sigma_0 \nonumber \\
	&\phantom{xx}{}- \bigg(\eps\theta_\eps[\bm{a}](\M_0) 
	+ \theta_\eps[\tilde{a}_0](\bm{\M})
	- 2\bm{\M}\times_{\#_{\rm even}}\bm{a}\bigg)\cdot\sigma. \label{3.TM2}
\end{align}
Furthermore, we replace $W$ in \eqref{2.TW} by $\T\M$, giving
\begin{align*}
  \T\T\M &= \big(p\cdot\na_x(\T\M)_0 + \alpha\eps\na_x^\perp\cdot\T\bm{\M}
	- \theta_\eps[V](\T\M)_0\big)\sigma_0 \\
	&\phantom{xx}{}+ \big(p\cdot\na_x\T\bm{\M} + \alpha\eps\na_x^\perp(\T\M)_0
	- 2\alpha p^\perp\times\T\bm{\M} - \theta_\eps[V]\T\bm{M}\big)\cdot\bm{\sigma}.
\end{align*}
Next, we integrate this expression with respect to $p$. The $\sigma_0$-component 
becomes, using the decomposition \eqref{3.TM2}, 
\begin{align*}
  \langle(\T\T\M)_0\rangle &= -\int_{\R^2}p\cdot\na_x\bigg(
	\theta_\eps[\tilde{a}_0](\M_0)
	+ \eps\sum_{j=1}^3\theta_\eps[a_j](\M_j)\bigg)dp \\
	&\phantom{xx}{}- \alpha\eps\int_{\R^2}\na_x^\perp\cdot\big(
	\eps\theta_\eps[\bm{a}](\M_0) + \theta_\eps[\tilde{a}_0](\bm{\M}) 
	- 2\bm{\M}\times_{\#_{\rm even}}\bm{a}\big)dp \\
	&\phantom{xx}{}+ \int_{\R^2}\theta_\eps[V]\bigg(\theta_\eps[\tilde{a}_0](\M_0)
	+ \eps\sum_{j=1}^3\theta_\eps[a_j](\M_j)\bigg)dp.
\end{align*}
In view of \eqref{moyal.int}, \eqref{2.theta2}, and $\langle\M_j\rangle=\eps n_j$, 
the first integral equals $\diver_x(n_0\na_x \tilde{a}_0 + \eps^2\bm{n}
\cdot\na_x\bm{a})$, while the second integral becomes
$2\alpha\eps^2\na_x^\perp\cdot(\bm{n}\times\bm{a})$, and the third integral
vanishes. Recalling that $\tilde{a}_0=a_0+V$, we infer that
\begin{equation}\label{3.T2}
  \langle(\T\T\M)_0\rangle
	= \diver_x\big(n_0\na_x a_0 + n_0\na_x V + \eps^2\bm{n}\cdot\na_x\bm{a}\big)
	+ 2\alpha\eps^2\na_x^\perp\cdot(\bm{n}\times\bm{a}).
\end{equation}
In a similar way, we compute the $\bm{\sigma}$-component of $\langle\T\T\M\rangle$:
\begin{align}
  \langle\T\T\bm{\M}\rangle &= -\int_{\R^2}p\cdot\na_x\big(\eps\theta_\eps[\bm{a}](\M_0)
	+ \theta_\eps[\tilde{a}_0](\bm{\M}) - 2\bm{\M}\times_{\#_{\rm even}}\bm{a}
	\big)dp \label{3.T3} \\
	&\phantom{xx}- \alpha\eps\int_{\R^2}\na_x^\perp\bigg(\theta_\eps[\tilde{a}_0]
	(\M_0) + \eps\sum_{j=1}^3\theta_\eps[a_j](\M_j)\bigg)dp \nonumber \\
	&\phantom{xx}{}+ 2\alpha\int_{\R^2}p^\perp\times\big(\eps\theta_\eps[\bm{a}](\M_0)
	+ \theta_\eps[\tilde{a}_0](\bm{\M}) - 2\bm{\M}\times_{\#_{\rm even}}\bm{a}
	\big)dp \nonumber \\
	&\phantom{xx}{}+ \int_{\R^2}\theta_\eps[V]\big(\eps\theta_\eps[\bm{a}](\M_0)
	+ \theta_\eps[\tilde{a}_0](\bm{\M}) + 2\bm{\M}\times_{\#_{\rm even}}\bm{a}
	\big)dp \nonumber \\
	&= \eps\diver_x\big(n_0\na_x\bm{a} + \bm{n}\na_x a_0 + \bm{n}\na_x V\big)
	+ 2\na_x\cdot(\bm{J}^T\times\bm{a}) \nonumber \\
	&\phantom{xx}{}- 2\alpha\eps\big(n_0\na_x^\perp\times\bm{a} + \na_x^\perp(a_0+V)
	\times\bm{n}\big) + 4\alpha\big(\bm{a}\langle p^\perp\cdot\bm{\M}\rangle
	+ \bm{J}^T\bm{a}^\perp\big), \nonumber
\end{align}
where $\bm{J}^T_k = \langle p_k\M\rangle$.

We turn now to the last term $\langle\T\pa_t\M\rangle$. 
Identity \eqref{3.TM} shows that
\begin{equation}\label{3.T4}
  \langle\T\pa_t\M\rangle = \pa_t\langle\T\M\rangle
	= 2\eps(\pa_t\bm{n}\times\bm{a} + \bm{n}\times\pa_t\bm{a})\cdot\bm{\sigma}.
\end{equation}
It remains to compute $\pa_t\bm{n}$ and $\pa_t\bm{a}$. By \eqref{3.TM} again,
$$
  \pa_t n_0\sigma_0 + \eps\pa_t\bm{n}\cdot\bm{\sigma} = \pa_t N 
	= -\langle\T\M\rangle + O(\tau) 
	= -2\eps(\bm{n}\times\bm{a})\cdot\bm{\sigma} + O(\tau),
$$
and thus $\pa_t^0 n_0=0$ and $\pa_t^0\bm{n}=-2\bm{n}\times\bm{a}$ 
at first order in $\tau$. We can write
$$
  \pa_t^0\bm{a} = \sum_{i=0}^3\frac{\delta\bm{a}}{\delta n_i}\pa_t^0 n_i
	= -2\sum_{i=1}^3\frac{\delta\bm{a}}{\delta n_i}(\bm{n}\times\bm{a})_i,
$$
where $\delta\bm{a}/\delta n_i$ denotes the variational derivative of $\bm{a}$.
Collecting expressions \eqref{3.TM}--\eqref{3.T4} finishes the proof of Theorem \ref{thm.full}.


\section{Derivation of the semiclassical quantum model}\label{sec.semi}


First, we expand $\M(N)$ in terms of $\eps$. 

\begin{proposition}\label{prop.max}
Let $\M(N)$ be the quantum Maxwellian defined in Theorem \ref{thm.QM}. Then
\begin{align*}
  \M(N) &= \exp(h_0)\sigma_0 + \eps\exp(h_0)\bm{h}_1\cdot\bm{\sigma} \\
	&\phantom{xx}{}+ \frac{\eps^2}{8}\exp(h_0)\bigg(\Delta a_0 + \frac13\big(|\na a_0|^2
	- p^T(\na\otimes\na a_0)p\big) + 4|\bm{h}_1|^2\bigg)\sigma_0 \\
	&\phantom{xx}{}+ \frac{\eps^3}{24}\exp(h_0)\Big(\big(3\Delta a_0 + |\na a_0|^2
	- p^T(\na\otimes\na a_0)p + 4|\bm{h}_1|^2\big)\bm{h}_1 \\
	&\phantom{xx}{}+ 3\Delta\bm{a} - 12\alpha\na^\perp\times\bm{a}
	+ 2\na\bm{a}\cdot\na a_0 - p^T(\na\otimes\na\bm{a})p 
	+ 2\alpha\na^\perp(\na a_0\cdot p) \\
	&\phantom{xx}{}+ 4\big((\na\bm{a})p - \alpha\na^\perp a_0\big)\times\bm{h}_1
	\Big)\cdot\bm{\sigma} + O(\eps^4),
\end{align*}
recalling that $h_0=-\tfrac12|p|^2+a_0$, $a_0=\tilde{a}-V$, and
$\bm{h}_1=\bm{a}-\alpha p^\perp$.
\end{proposition}

We need an expansion up to order $\eps^3$ since the nonlocal model in
Theorem \ref{thm.full} contains a term of order $\eps^{-1}$. 

\begin{proof}
We introduce the function
$$
  g(\beta) = \Exp\big(\beta(h_0\sigma_0 + \eps\bm{h}_1\cdot\bm{\sigma})\big),
	\quad\beta\ge 0.
$$
We see from \eqref{2.HepsA} that
the quantum Maxwellian corresponds to $\M(N)=g(1)$. The variable $\beta$
can be interpreted as the inverse temperature, and $\beta=1$ means that the
temperature of the systems equals the thermal temperature. 
Lemma \ref{lem.hash} implies that
\begin{align}
  \pa_\beta g(\beta) &= \pa_\beta\big\{\mathcal{W}\big[
	\exp\big(\beta(-\mathcal{H}_\eps+\tilde{A})\big)\big]\big\}
	= \mathcal{W}\big[(-\mathcal{H}_\eps+\tilde{A})
	\exp(-\mathcal{H}_\eps+\tilde{A})\big] \label{4.g} \\
  &= \mathcal{W}(-\mathcal{H}_\eps+\tilde{A})\#\mathcal{W}
	\big[\exp\big(\beta(-\mathcal{H}_\eps+\tilde{A})\big)\big]
	= (-H_\eps+\tilde{A})\# g(\beta) \nonumber
\end{align}
for $\beta>0$ and $g(0)=\sigma_0$.
Introducing the semiclassical expansions $g(\beta)=\sum_{k=0}^\infty\eps^k 
g^{(k)}(\beta)$ on the left-hand side, inserting the semiclassical expansion of the
Moyal product on the right-hand side (Lemma \ref{lem.moyal}), 
and identifying the corresponding order
of $\eps$, we obtain a system of recursive ordinary differential equations 
for $g^{(k)}$,
\begin{equation}\label{4.aux}
  \pa_\beta g^{(k)}(\beta) = \sum_{\ell=0}^k h_0\sigma_0\#_\ell g^{(k-\ell)}(\beta)
	+ \sum_{\ell=0}^{k-1}(\bm{h}_1\cdot\bm{\sigma})\#_\ell g^{(k-\ell-1)}(\beta),
	\quad k\ge 0,
\end{equation}
with the initial conditions $g^{(0)}(0)=\sigma_0$ and $g^{(k)}(0)=0$ for $k\ge 1$.
Recalling that the zeroth-order Moyal product is just the ordinary matrix 
multiplication, we find for $k=0$ that
$$
  \pa_\beta g^{(0)}(\beta) = h_0 g^{(0)}(\beta), \quad\beta>0, \quad
	g^{(0)}(0)=\sigma_0,
$$
with the solution $g^{(0)}(\beta)=\exp(\beta h_0)\sigma_0$. 
(Note that this solution differs from the corresponding one in
Appendix \ref{app.QM} since there, the function $\bm{h}_1$ contains an additional 
term of order one.) For $k\ge 1$, \eqref{4.aux} becomes
$$
  \pa_\beta g^{(k)}(\beta) = h_0 g^{(0)}\sigma_0
	+ \sum_{\ell=0}^{k-1}\big(h_0\sigma_0\#_\ell g^{(k-\ell)} 
	+ (\bm{h}_1\cdot\bm{\sigma})\#_\ell g^{(k-\ell-1)}\big), \quad g^{(k)}(0)=0.
$$
We show in Appendix \ref{app.g} that
\begin{align}
  g^{(1)}(\beta) &= \beta\exp(\beta h_0)\bm{h}_1\cdot\bm{\sigma}, \label{g1} \\
	g^{(2)}(\beta) &= \frac{\beta^2}{8}\exp(\beta h_0)\bigg(\Delta a_0
	+ \frac{\beta}{3}\big(|\na a_0|^2 - p^T(\na\otimes\na a_0)p\big)
	+ 4|\bm{h}_1|^2\bigg)\sigma_0, \label{g2} \\
	g^{(3)}(\beta) &= \frac{\beta^2}{24}\exp(\beta h_0)\bigg(\big(3\beta\Delta a_0
	+ \beta^2(|\na a_0|^2 - p^T(\na\otimes\na a_0)p) + 4\beta|\bm{h}_1|^2\big)\bm{h}_1 
	\label{g3} \\
	&\phantom{xx}{}+ 3\Delta\bm{a} - 12\alpha\na^\perp\times\bm{a}
	+ \beta\big(2\na\bm{a}\cdot\na a_0 - p^T(\na\otimes\na\bm{a})p
	+ 2\alpha\na^\perp(\na a_0\cdot p)\big) \nonumber \\
	&\phantom{xx}{} + 4\beta\big((\na\bm{a}) p - \alpha\na^\perp a_0\big)\times\bm{h}_1
	\bigg)\cdot\bm{\sigma}. \nonumber
\end{align}
The result follows after substituting the previous expressions into 
$\M(N)=\sum_{k=1}^3 \eps^k g^{(k)}(1)+O(\eps^4)$ and collecting the terms.
\end{proof}

Expressions \eqref{g1}--\eqref{g3} 
correspond to the expansion of $\M$ as an explicit function
of $\eps$, i.e.\ $g^{(k)}(1)|_{\eps=0} = (1/m!)(\pa^k\M/\pa\eps^k)|_{\eps=0}$. 
However, $\M$ depends on $\eps$ also through its dependence on
$\tilde{A}$. Thus, we need to expand $\tilde{A}$ or, equivalently,
$a_0$ and $\bm{a}$ in terms of $\eps$. To this end, we expand
$$
  \M = \sum_{k=0}^\infty\eps^k \M^{(k)}, \quad
	a_j = \sum_{k=0}^\infty\eps^k a_j^{(k)}, \quad j=0,1,2,3.
$$
We Taylor-expand the left-hand side with respect to $\eps$ 
and identify the expressions with the corresponding orders of $\eps$ from
the right-hand side:
\begin{align}
  \M^{(0)} &= \M\big|_{\eps=0}, \quad
	\M^{(1)} = \frac{\pa\M}{\pa\eps}\bigg|_{\eps=0} + \sum_{j=0}^3\frac{\pa\M}{\pa a_j}
	\bigg|_{\eps=0}a_j^{(1)}, \label{3.M1} \\
	2\M^{(2)} &= \frac{\pa^2\M}{\pa\eps^2}\bigg|_{\eps=0}
	+ 2\sum_{j=0}^3\bigg(\frac{\pa^2\M}{\pa\eps\pa a_j}\bigg|_{\eps=0}a_j^{(1)}
	+ \frac{\pa\M}{\pa a_j}\bigg|_{\eps=0}a_j^{(2)}\bigg)
	+ \sum_{j,k=0}^3\frac{\pa^2\M}{\pa a_j\pa a_k}\bigg|_{\eps=0}a_j^{(1)}a_k^{(1)}, 
	\label{3.M2} \\
	6\M^{(3)} &= \frac{\pa^3\M}{\pa\eps^3}\bigg|_{\eps=0}
	+ 3\sum_{j=0}^3\bigg(\frac{\pa^3\M}{\pa\eps^2\pa a_j}\bigg|_{\eps=0}a_j^{(1)}
	+ 2\frac{\pa^2\M}{\pa\eps\pa a_j}\bigg|_{\eps=0}a_j^{(2)}
	+ 2\frac{\pa\M}{\pa a_j}\bigg|_{\eps=0}a_j^{(3)}\bigg) \label{3.M3} \\
	&\phantom{xx}{}+ \sum_{j,k=0}^3\bigg(3\frac{\pa^3\M}{\pa\eps\pa a_j\pa a_k}
	\bigg|_{\eps=0}a_j^{(1)}a_k^{(1)} + 2\frac{\pa^2\M}{\pa a_j\pa a_k}\bigg|_{\eps=0}
	(2a_j^{(2)}a_k^{(1)} + a_j^{(1)}a_k^{(2)})\bigg) \nonumber \\
	&\phantom{xx}{} + \sum_{j,k,\ell=0}^3\frac{\pa^3\M}{\pa a_j\pa a_k\pa a_\ell}
	\bigg|_{\eps=0}a_j^{(1)}a_k^{(1)}a_\ell^{(1)}. \nonumber
\end{align}
The $j$th-order of the Lagrange multiplier $a_i^{(j)}$ is determined by
identifying the orders in the constraint $\langle\M\rangle = n_0\sigma_0
+\eps\bm{n}\cdot\bm{\sigma}$:
\begin{equation}\label{4.Mj}
  \langle\M^{(0)}\rangle = n_0\sigma_0, \quad
	\langle\M^{(1)}\rangle = \bm{n}\cdot\bm{\sigma}, \quad
	\langle\M^{(2)}\rangle = \langle\M^{(3)}\rangle = 0.
\end{equation}
This leads to the following result.

\begin{lemma}\label{lem.a}
The semiclassical expansion of the Lagrange multipliers $a_0$ and $\bm{a}$ reads as
\begin{align*}
  a_0 &= \log\frac{n_0}{2\pi} - \eps^2\bigg\{\frac{1}{12}\bigg(\frac{\Delta n_0}{n_0}
	- \frac{|\na n_0|^2}{2n_0^2}\bigg) + \frac12\bigg|\frac{\bm{n}}{n_0}\bigg|^2
	+ \alpha^2\bigg\} + O(\eps^4), \\
	\bm{a} &= \frac{\bm{n}}{n_0} + \frac{\eps^2}{3}\bigg\{\frac{\bm{n}}{4n_0}\bigg(
	\frac{\Delta n_0}{n_0} - \bigg|\frac{\na n_0}{n_0}\bigg|^2 
	+ 4\bigg|\frac{\bm{n}}{n_0}\bigg|^2 + 8\alpha^2\bigg)
	+ \alpha^2\frac{\bm{n}^{\perp\perp}}{n_0} \\
	&\phantom{xx}{}- \frac{1}{4}\bigg(\frac{\Delta\bm{n}}{n_0}
	- \frac{\na\bm{n}}{n_0}\cdot\frac{\na n_0}{n_0}\bigg)
	+ \frac{\alpha}{2n_0}\bigg(4\na^\perp\times\bm{n} + \frac{\na^\perp n_0}{n_0}
	\times\bm{n}\bigg)\bigg\} + O(\eps^3), 
\end{align*}
recalling that $\bm{n}^{\perp\perp}:=(-n_1,-n_2,0)^T$.
\end{lemma}

\begin{proof}
We compute the coefficients $a_0^{(j)}$ for $j=0,1,2,3$ and
$\bm{a}^{(j)}$ for $j=0,1,2$ using \eqref{4.Mj}. The first condition leads to
$$
  n_0 = \big\langle\M^{(0)}_0\big|_{\eps=0}\big\rangle
	= \big\langle g_0^{(0)}\big|_{\eps=0}\big\rangle
  = \big\langle \exp(-\tfrac12|p|^2+a_0^{(0)})\big\rangle = 2\pi\exp(a_0^{(0)}),
$$
since $\langle\exp(-\frac12|p|^2)\rangle=2\pi$,
which allows us to identify $a_0^{(0)}=\log(n_0)-\log(2\pi)$. 

Next, we observe that the other derivatives of $\M$ are given by
\begin{align*}
  & \frac{\pa\M}{\pa a_0}\bigg|_{\eps=0} = \frac{\pa^2\M}{\pa^2 a_0^2}\bigg|_{\eps=0}
	= \exp(h_0^{(0)})\sigma_0, \\
	& \frac{\pa\M}{\pa a_j}\bigg|_{\eps=0} 
	= \frac{\pa^2\M}{\pa a_j\pa a_k}\bigg|_{\eps=0} = 0
	\quad\mbox{for }j\neq 0,\ k\neq 0, \\
	& \frac{\pa^2\M}{\pa\eps\pa a_0}\bigg|_{\eps=0} = \exp(h_0^{(0)})\bm{h}_1^{(0)}
	\cdot\bm{\sigma}, \quad
	\frac{\pa^2\M}{\pa\eps\pa a_j}\bigg|_{\eps=0} = \exp(h_0^{(0)})\sigma_j\quad
	\mbox{for }j\neq 0,
\end{align*}
where we have set $h_0^{(0)}:=-\frac12|p|^2+a_0^{(0)}$ and 
$\bm{h}_1^{(0)} := \bm{a}^{(0)} - \alpha p^\perp$. By \eqref{3.M1} and \eqref{4.Mj}, 
this yields for $\M^{(1)}$:
\begin{align*}
  \bm{n}\cdot\bm{\sigma} &= \langle\M^{(1)}\rangle 
	= \bigg\langle\frac{\pa\M}{\pa\eps}\bigg|_{\eps=0}\bigg\rangle
	+ \sum_{j=0}^3\bigg\langle\frac{\pa\M}{\pa a_j}\bigg|_{\eps=0}a_j^{(1)}
	\bigg\rangle = \big\langle\exp(h_0^{(0)})\big(a_0^{(1)}\sigma_0
	+ \bm{h}_1^{(0)}\cdot\bm{\sigma}\big)\big\rangle \\
	&= 2\pi\exp(a_0^{(0)})\big(a_0^{(1)}\sigma_0 + \bm{a}^{(0)}\cdot\bm{\sigma}\big)
	= n_0\big(a_0^{(1)}\sigma_0 + \bm{a}^{(0)}\cdot\bm{\sigma}\big).
\end{align*}
Identifying the Pauli coefficients, we infer that
$a_0^{(1)} = 0$ and $\bm{a}^{(0)} = \bm{n}/n_0$.

For $\M^{(2)}$, we use \eqref{3.M2} in $\langle\M^{(2)}\rangle=0$, and
insert the expressions for the partial derivatives of $\M$. A tedious but elementary
computation leads to
$$
  \bigg(\frac18\Delta a_0^{(0)} + \frac{1}{24}\big(|\na a_0^{(0)}|^2 
	- \Delta a_0^{(0)}\big) + \frac12\big(|\bm{a}^{(0)}|^2 + 2\alpha^2\big)
	 + a_0^{(2)}\bigg)\sigma_0 + \bm{a}^{(1)}\cdot\bm{\sigma} = 0.
$$
It follows that $\bm{a}^{(1)}=\bm{0}$ and, inserting
$a_0^{(0)}=\log(n_0)-\log(2\pi)$ and $\bm{a}^{(0)}=\bm{n}/n_0$,
\begin{equation}\label{3.a02}
  a_0^{(2)} = -\frac{1}{12}\Delta\log n_0 - \frac{1}{24}|\na\log n_0|^2
	- \frac{|\bm{n}|^2}{2n_0^2} - \alpha^2.
\end{equation}

It remains to evaluate $\langle\M^{(3)}\rangle=0$. Our previous results allow us
to simplify expansion \eqref{3.M3}:
$$
  0 = \bigg\langle g^{(3)}(1)\big|_{\eps=0} + \sum_{j=0}^3\frac{\pa^2\M}{\pa\eps\pa a_j}
	\bigg|_{\eps=0}a_j^{(2)} + \frac{\pa\M}{\pa a_0}\bigg|_{\eps=0}a_0^{(3)}
	\bigg\rangle.
$$
The first two terms have a spinorial part only, while the third term has only 
a trace part. This gives $a_0^{(3)}=0$. It remains to calculate
\begin{align}
  0 &= \big\langle g^{(3)}(1)\big|_{\eps=0} + a_0^{(2)}\exp(h_0^{(0)})\bm{h}_1^{(0)}
	+ \exp(h_0^{(0)})\bm{a}^{(2)}\big\rangle \label{3.gg3} \\
  &= \big\langle g^{(3)}(1)\big|_{\eps=0}\big\rangle
	+ a_0^{(2)}\bm{n} + n_0\bm{a}^{(2)}, \nonumber
\end{align}
which in fact determines $\bm{a}^{(2)}$. 
A straightforward but again tedious computation shows that the first term equals
\begin{align}
  \big\langle g^{(3)}(1)\big|_{\eps=0}\big\rangle
	&= \frac{n_0}{12}\bigg\{\bigg(\Delta a_0^{(0)} + \frac12|\na a_0^{(0)}|^2
	+ 2(|\bm{a}^{(0)}|^2 + 2\alpha^2)\bigg)\bm{a}^{(0)} 
	- 4\alpha^2(\bm{a}^{(0)})^{\perp\perp}\bigg\} \label{3.g3} \\
	&\phantom{xx}{}+ \frac{n_0}{12}\big(\Delta\bm{a}^{(0)} + \na\bm{a}^{(0)}
	\cdot\na a_0^{(0)} - 2\alpha(4\na^\perp\times\bm{a}^{(0)} + \na^\perp a_0^{(0)})
	\times\bm{a}^{(0)}), \nonumber
\end{align}
where $(\bm{a}^{(0)})^{\perp\perp} = (-a_1^{(0)},-a_2^{(0)},0)^T$.
We differentiate $a_0^{(0)}=\log(n_0)-\log(2\pi)$ and $\bm{a}^{(0)}=\bm{n}/n_0$
with respect to $x$ and include the resulting expressions into \eqref{3.g3}. Then, using
expression \eqref{3.a02} for $a_0^{(2)}$, \eqref{3.gg3} allows us to compute
$\bm{a}^{(2)}$, eventually yielding
\begin{align}
  \bm{a}^{(2)} &= \frac{1}{12n_0}\bigg\{\bigg(\frac{\Delta n_0}{n_0}
	- \bigg|\frac{\na n_0}{n_0}\bigg|^2 + 4\bigg|\frac{\bm{n}}{n_0}\bigg|^2
	+ 8\alpha^2\bigg)\bm{n} + 4\alpha^2\bm{n}^{\perp\perp}\bigg\} \label{3.a2} \\
	&\phantom{xx}{}- \frac{1}{12}\bigg(\frac{\Delta\bm{n}}{n_0} - \frac{\na\bm{n}}{n_0}
	\cdot\frac{\na n_0}{n_0}\bigg) + \frac{\alpha}{6n_0}\bigg(
	4\na^\perp\times\bm{n}+\frac{1}{n_0}\na^\perp n_0\times\bm{n}\bigg). \nonumber
\end{align}
This finishes the proof.
\end{proof}

For the proof of Theorem \ref{thm.semi}, we insert the expansions from 
Lemma \ref{lem.a} into the nonlocal quantum-spin model \eqref{1.n0}--\eqref{1.n}. 
We compute
\begin{align*}
  n_0\na a_0 &= n_0\na\big(a_0^{(0)} + \eps^2 a_0^{(2)}\big) + O(\eps^3) \\
	&= \na n_0 - \eps^2n_0\na\bigg(\frac{1}{12}\bigg(\frac{\Delta n_0}{n_0}
	- \frac{|\na n_0|^2}{2n_0^2}\bigg) + \frac12\bigg|\frac{\bm{n}}{n_0}\bigg|^2
	\bigg) + O(\eps^3) \\
	&= \na n_0 - \eps^2\bigg(\frac{n_0}{6}\na
	\frac{\Delta\sqrt{n_0}}{\sqrt{n_0}} + \frac{1}{n_0}\bm{n}\cdot\na\bm{n}
	- \bigg|\frac{\bm{n}}{n_0}\bigg|^2\na n_0\bigg) + O(\eps^3).
\end{align*}
We only need the zeroth order for $\na\bm{a}$ since it appears at order
$O(\eps^2)$. Then $\na\bm{a} = \na\bm{n}/n_0
- \bm{n}\na n_0/n_0^2 + O(\eps^2)$ and consequently,
$$
  \bm{n}\cdot\na\bm{a} = \frac{1}{n_0}\bm{n}\cdot\na\bm{n}
	- \bigg|\frac{\bm{n}}{n_0}\bigg|^2\na n_0 + O(\eps^2).
$$
Hence, neglecting terms of order $\alpha^m\eps^n$ with $m+n>2$, equation \eqref{1.n0} 
for the charge density becomes
\begin{align*}
  \pa_t n_0 &= \tau\diver\bigg\{n_0\na\log\frac{n_0}{2\pi} + n_0\na V
	- \frac{\eps^2}{12}n_0\na\bigg(\frac{\Delta n_0}{n_0} - \frac{|\na n_0|^2}{2n_0}
	\bigg) - \frac{\eps^2}{2}n_0\na\bigg|\frac{\bm{n}}{n_0}\bigg|^2
	+ \eps^2\bm{n}\cdot\na\frac{\bm{n}}{n_0}\bigg\} \\
	&= \tau\diver\bigg(\na n_0+n_0\na V - \frac{\eps^2}{6}n_0
	\na\frac{\Delta\sqrt{n_0}}{\sqrt{n_0}}\bigg).
\end{align*}

The computation for equation \eqref{1.n} for the spin vector is more involved.
We calculate the expansion for the terms of the first line of \eqref{1.n}, 
using Lemma \ref{lem.a} and only reporting the results:
\begin{align}
  -2\bm{n}\times\bm{a} &= \frac{\eps^2}{6}\frac{\bm{n}}{n_0}\times\bigg(
	\frac{\Delta\bm{n}}{n_0} - \frac{\na\bm{n}}{n_0}\cdot\frac{\na n_0}{n_0}\bigg)
	+ O(\eps^3), \label{4.ntimesa} \\
	n_0\na\bm{a}+\bm{n}\na a_0 &= \na\bm{n} + \frac{\eps^2}{12}\bm{n}\bigg(
	2\bigg|\frac{\na n_0}{n_0}\bigg|^2\frac{\na n_0}{n_0}
	- 4\frac{\bm{n}}{n_0}\cdot\frac{\na\bm{n}}{n_0}
	- \frac{\na n_0}{n_0}\frac{\Delta n_0}{n_0}
	- \frac{\na n_0(\na\otimes\na n_0)}{n_0^2}\bigg) \nonumber \\
	&\phantom{xx}{}- \frac{\eps^2}{12}\na\Delta\bm{n}
	+ \frac{\eps^2}{12}\na\bm{n}\bigg\{\frac{\na\otimes\na n_0}{n_0}
	+ \bigg(4\bigg|\frac{\bm{n}}{n_0}\bigg|^2 + \frac{\Delta n_0}{n_0}
	- \bigg|\frac{\na n_0}{n_0}\bigg|^2\bigg)\sigma_0\bigg\} \nonumber \\
	&\phantom{xx}{}+ \frac{\eps^2}{12}\bigg\{\bigg(\Delta\bm{n}
	- 2\na\bm{n}\cdot\frac{\na n_0}{n_0}\bigg)\frac{\na n_0}{n_0}
	+ \frac{\na n_0}{n_0}\na\otimes\na\bm{n}\bigg\} + O(\eps^3), \nonumber \\
	\frac{2}{\eps}\bm{J}^T\times\bm{a} &= 2\alpha
	\begin{pmatrix} n_3 & 0 & 0 \\ 0 & n_3 & 0 \\ -n_1 & -n_2 & 0 \end{pmatrix}
	+ \frac{\eps^2}{3n_0}\na\bm{n}\times\bm{n} + O(\alpha\eps^2+\eps^3). \nonumber 
\end{align}
The first two terms in the second line of \eqref{1.n} are of order $\alpha$
such that we only need to expand them up to first order. We obtain, up to an error
of order $O(\alpha\eps^2)$,
$$
  -2\alpha\tau\big(n_0\na^\perp\times\bm{a}^{(0)} 
	+ \na^\perp(a_0^{(0)}+V)\times\bm{n}\big) 
	= -2\alpha\tau(\na^\perp\times\bm{n}+\na^\perp V\times\bm{n}).
$$
Using $p^\perp\cdot\bm{h}_1=p^\perp\cdot\bm{a}-\alpha|p^\perp|^2$ and
$\bm{J}^T\bm{a}^\perp = -\alpha\eps\bm{n}^{\perp\perp} + O(\alpha\eps^2)$,
the last part of the second line of \eqref{1.n} becomes
\begin{align*}
  4\tau&\frac{\alpha}{\eps}\big(\bm{a}\langle p^\perp\cdot\M(N)\rangle 
	+ \bm{J}^T\bm{a}^\perp\big)
	= 4\alpha\tau\bigg(\frac{\bm{n}}{n_0}\langle\exp(h_0)p^\perp
	\cdot\bm{h}_1\rangle - \alpha\bm{n}^{\perp\perp}\bigg) + O(\alpha\eps^2) \\
	&= -4\alpha^2\tau(2\bm{n}+\bm{n}^{\perp\perp}) + O(\alpha\eps^2).
\end{align*}
Recalling that $\bm{n}\times\bm{a} = O(\eps^2)$, the first term in the last line
of \eqref{1.n} is of order $O(\eps^3)$, i.e.\ $2\eps(\bm{n}\times\bm{a})\times\bm{a}
=O(\eps^3)$, and will be neglected.

It remains to expand the last term in the last line of \eqref{1.n},
$\bm{n}\times\pa_t^{0}\bm{a}$, where $\pa_t^0\bm{a}$
is the lowest-order term of $\pa_t\bm{a}$ with respect to $\tau$. We expand it
with respect to $\eps$:
$$
  \pa_t^0\bm{a} = \pa_t^0\bm{a}^{(0)} + \eps^2\pa_t^0\bm{a}^{(2)} + O(\eps^3).
$$
Equations \eqref{1.n0} and \eqref{1.n} show that $\pa_t^0n_0 = 0$,
$\pa_t^0\bm{n}=-2\bm{n}\times\bm{a}$ and therefore,
\begin{align*}
  \pa_t^0\bm{a}^{(0)} &= \pa_t^0\bigg(\frac{\bm{n}}{n_0}\bigg)
	= \frac{\pa_t^0\bm{n}}{n_0} - \frac{\bm{n}}{n_0^2}\pa_t^0 n_0
	= -\frac{2}{n_0}\bm{n}\times\bm{a} \\
	&= -\frac{2}{n_0}\bm{n}\times(\bm{a}^{(0)}+\eps^2 \bm{a}^{(2)}) + O(\eps^3)
	= -\frac{2\eps^2}{n_0}\bm{n}\times\bm{a}^{(2)} + O(\eps^3),
\end{align*}
since $\bm{n}\times\bm{a}^{(0)}=\bm{n}\times\bm{n}/n_0=0$,
and $\bm{a}^{(2)}$ is given by \eqref{3.a2}:
\begin{align*}
  & \bm{a}^{(2)} = f(n_0,\bm{n})\frac{\bm{n}}{n_0} 
	- \frac{1}{12}\frac{\Delta\bm{n}}{n_0} 
	+ \frac{1}{12}\frac{\na\bm{n}}{n_0}\cdot\frac{\na n_0}{n_0}, \\
	& \mbox{where }f(n_0,\bm{n}) := \frac13\bigg|\frac{\bm{n}}{n_0}\bigg|^2
	+ \frac{1}{12}\frac{\Delta n_0}{n_0} - \frac{1}{12}\bigg|\frac{\na n_0}{n_0}\bigg|^2.
\end{align*}
Because of $\bm{n}\times\bm{n}=\bm{0}$, some terms cancel in $\bm{n}\times\bm{a}^{(2)}$,
and we end up with
$$
  \pa_t^0\bm{a}^{(0)} = \frac{\eps^2}{6n_0}\bm{n}\times\bigg(
	\frac{\Delta n_0}{n_0} - \frac{\na\bm{n}}{n_0}\cdot\frac{\na n_0}{n_0}\bigg)
	+ O(\eps^3).
$$
Differentiating $\bm{a}^{(2)}$ yields
$$
  \pa_t^0\bm{a}^{(2)} = \pa_t^0 f(n_0,\bm{n})\frac{\bm{n}}{n_0}
	+ \frac{f(n_0,\bm{n})}{n_0}\pa_t^0\bm{n}
	- \frac{1}{12}\pa_t^0\bigg(\frac{\Delta n_0}{n_0} 
	- \frac{\na\bm{n}}{n_0}\cdot\frac{\na n_0}{n_0}\bigg).
$$
Since we only need the cross product $\bm{n}\times\pa_t^0\bm{a}$, the first term,
which is parallel to $\bm{n}$, vanishes. Moreover, the second term 
$\pa_t^0\bm{n}=-2\bm{n}\times\bm{a}=-2\bm{n}\times\bm{a}^{(0)}+O(\eps^2)
=O(\eps^2)$ can be neglected, as it is already of higher order in $\eps$.
The same conclusion holds true for the third term:
$$
  \pa_t^0\bigg(\frac{\Delta n_0}{n_0} 
	- \frac{\na\bm{n}}{n_0}\cdot\frac{\na n_0}{n_0}\bigg)
  = -\frac{2}{n_0}\Delta(\bm{n}\times\bm{a}^{(0)}) + \frac{2}{n_0}\na(\bm{n}\times
	\bm{a}^{(0)}) + O(\eps^2) = O(\eps^2).
$$
It follows that $\pa_t^0\bm{a}^{(2)}=\pa_t^0 f(n_0,\bm{n})(\bm{n}/n_0)+O(\eps^2)$.
Summarizing these results, we end up with
$$
  \bm{n}\times\pa_t^0\bm{a}
	= \bm{n}\times\pa_t^0(\bm{a}^{(0)} + \eps^2\bm{a}^{(2)}) + O(\eps^3)
	= \frac{\eps^2}{6n_0}\bm{n}\times\bigg\{\bm{n}\times\bigg(\frac{\Delta n_0}{n_0} 
	- \frac{\na\bm{n}}{n_0}\cdot\frac{\na n_0}{n_0}\bigg)\bigg\} + O(\eps^3).
$$
Collecting these expressions, we see that equations \eqref{1.n0}--\eqref{1.n}
reduce, up to order $O(\alpha^m\eps^n)$ with $m+n > 2$, to the local model
\eqref{1.n0loc}--\eqref{1.nloc}.


\begin{appendix}

\section{Proof of Theorem \ref{thm.QM}}\label{app.QM}

We split the proof into three steps. First, we show a weaker result than
stated in Theorem \ref{thm.QM}, namely that the spin component $\tilde{\bm{a}}$ of 
$\tilde{A} = a_0\sigma_0+\tilde{\bm{a}}\cdot\bm{\sigma}$
is possibly of order one. Then we compute the leading order of the semiclassical
expansion of the quantum Maxwellian and show that in fact
$\tilde{\bm{a}}=\eps\bm{a}$ is of order $\eps$.

{\em Step 1.} Let $N=n_0\sigma_0+\eps\bm{n}\cdot\bm{\sigma}$ be given.
Our aim is to show that if the quantum maximum entropy problem has a solution
then it is of the form $(\M(N)) = \Exp(-H_\eps+\tilde{A})$,
where $\tilde{A}=\tilde{a}_0\sigma_0+\tilde{\bm{a}}\cdot\bm{\sigma}$. 
By construction, the solution $\M(N) = \M_0\sigma_0+\bm{\M}\cdot\bm{\sigma}$ satisfies
$$
  n_0 = \langle\M_0\rangle, \quad \eps n_j = \langle\M_j\sigma_j\rangle, 
	\quad j=1,2,3.
$$
The proof follows the corresponding proofs in the literature; see 
\cite{DMR05,DeRi03}. Our constrained minimization problem is equivalent to 
the saddle-point problem
$$
  \mathcal{E}(\M(N)) = \min_W\max_{\tilde{A}}\mathcal{L}(W,\tilde{A})
	= \max_{\tilde{A}}\min_W\mathcal{L}(W,\tilde{A}),
$$
where the variational functional equals
$$
  \mathcal{L}(W,\tilde{A}) = \mathcal{E}(W) - \mathcal{K}(W), 
	\quad\mathcal{K}(W) = \operatorname{tr}\int_{\R^2}
	\tilde{A}(x)(\langle W\rangle - N)dx
$$
and $\mathcal{E}(W)$ is given in \eqref{2.E}.
Reformulating \cite[Lemma 3.3]{DeRi03} in terms of Wigner functions, we see that
the G\^ateaux derivative of the free energy with respect to $W$
in the direction of $\xi$ is given by
$$
  \delta_W\mathcal{E}(W,\tilde{A};\xi) = \operatorname{tr}\int_{\R^2}\int_{\R^2}
	(\Log(W)+H_\eps)\xi dxdp.
$$
Since $\tilde{A}$ and $N$ do not depend on $W$, the G\^ateaux derivative of
$\mathcal{K}$ is
$$
  \delta_W\mathcal{K}(W,\tilde{A};\xi) = 
	\operatorname{tr}\int_{\R^2}\int_{\R^2}\tilde{A}\langle\xi\rangle dx.
$$
Thus, the Euler--Lagrange equation associated to the problem
$\min_W\mathcal{L}(W,\tilde{A})$ becomes
$$
  \operatorname{tr}\int_{\R^2}\int_{\R^2}(\Log W + H_\eps - \tilde{A})\xi dxdp = 0
$$
for all variations $\xi$. This implies that $\Log W + H_\eps - \tilde{A}=0$
and hence $\M(N):=W=\Exp(-H_\eps+\tilde{A})$. We compute the G\^ateaux derivative of
$\mathcal{L}(W,\tilde{A})$ with respect to $\tilde{A}$, using \eqref{2.pauli}:
$$
  0 = \delta_{\tilde{A}}\mathcal{L}(W,\tilde{A};\xi)
	= \operatorname{tr}\int_{\R^2}\xi(\langle W\rangle - N)dx 
	= 2\int_{\R^2}\big(\xi_0(\langle w_0\rangle - n_0)
	+ \bm{\xi}\cdot(\langle\bm{w}\rangle-\eps\bm{n})\big)dx,
$$
where the variations are given by $\xi=\xi_0\sigma_0+\bm{\xi}\cdot\bm{\sigma}$.
This immediately gives $n_0=\langle w_0\rangle$ and $\eps\bm{n}=\langle\bm{w}\rangle$.

{\em Step 2.} We formulate $-H_\eps+\tilde{A}$ in terms of the Pauli components:
\begin{align*}
  & {-H_\eps} + \tilde{A} = h_0\sigma_0 + \bm{h}_1\cdot\bm{\sigma}, \quad\mbox{where} \\
	& h_0(x,p) = -\tfrac12|p|^2 + a_0(x), \quad a_0(x) = \tilde{a}_0(x)-V(x), \quad
	\bm{h}_1(x,p) = \bm{a}(x) - \eps\alpha p^\perp.
\end{align*}
(Note that the definition of $\bm{h}$ is slightly different from that one in
\eqref{2.HepsA} since we do not know at this point that $\tilde{\bm{a}}$ is
of order $\eps$.)
Let $\tilde{A}^{(0)}=\tilde{a}_0^{(0)}\sigma_0+\bm{a}^{(0)}\cdot\bm{\sigma}$
be the leading order of the semiclassical expansion of $\tilde{A}$ with respect
to $\eps$. We claim that the leading order of $\M(N)$ is given by
\begin{equation}\label{a.M0}
  \M^{(0)}(N) = \exp(h_0^{(0)})\bigg(\cosh|\bm{a}^{(0)}|\sigma_0
	+ \sinh|\bm{a}^{(0)}|\frac{\bm{a}^{(0)}}{|\bm{a}^{(0)}|}\cdot\bm{\sigma}\bigg),
\end{equation}
where $h_0^{(0)}=-|p|^2/2+a_0^{(0)}$ and $a_0^{(0)}=\tilde{a}_0^{(0)}-V(x)$.
If $\bm{a}_0^{(0)}=\bm{0}$, we set $\M^{(0)}(N)=\exp(h_0^{(0)})\sigma_0$.

The proof of \eqref{a.M0} is similar to the proof of Proposition \ref{prop.max}.
We have shown in \eqref{4.g} that $g(\beta)=\Exp(\beta(-H_\eps+\tilde{A}))$ 
satisfies the differential equation
$$
  \pa_\beta g(\beta) = (h_0\sigma_0 + \bm{h}_1\cdot\bm{\sigma})\# g(\beta),
	\quad \beta>0, \quad g(0) = \sigma_0.
$$
According to Lemma \ref{lem.moyaleps}, this equation becomes at lowest order
$$
  \pa_\beta g^{(0)}(\beta) = (h_0^{(0)}\sigma_0 + \bm{a}^{(0)}\cdot\bm{\sigma})
	g^{(0)}(\beta), \quad \beta>0, \quad g(0)=\sigma_0.
$$
To solve this differential equation, we remove the first term on the right-hand
side by introducing the function $f(\beta)=\exp(-\beta h_0^{(0)})g^{(0)}(\beta)$,
which solves
$$
  \pa_\beta f(\beta) = \bm{a}^{(0)}\cdot\bm{\sigma}f(\beta), \quad \beta>0,
	\quad f(0)=\sigma_0.
$$
The solution is the matrix exponential 
$f(\beta)=\exp(\beta\bm{a}^{(0)}\cdot\bm{\sigma})$. Recalling that
$$
  \bm{a}^{(0)}\cdot\bm{\sigma} = \begin{pmatrix}
	a_3^{(0)} & a_1^{(0)}-\ii a_2^{(0)} \\ a_1^{(0)} + \ii a_2^{(0)} & -a_3^{(0)}
	\end{pmatrix},
$$
direct calculations give
$(\bm{a}^{(0)}\cdot\bm{\sigma})^{2k}=|\bm{a}^{(0)}|^{2k}\sigma_0$ and
$(\bm{a}^{(0)}\cdot\bm{\sigma})^{2k+1}=|\bm{a}^{(0)}|^{2k}\bm{a}^{(0)}\cdot\bm{\sigma}$
for all $k\in\N$. Therefore,
$$
  f(\beta) = \sum_{k=0}^\infty\frac{\beta^k}{k!}(\bm{a}^{(0)}\cdot\bm{\sigma})^k
	= \cosh(\beta|\bm{a}^{(0)}|)\sigma_0
	+ \sinh(\beta|\bm{a}^{(0)}|)\frac{\bm{a}^{(0)}}{|\bm{a}^{(0)}|}\cdot\bm{\sigma}
$$
if $\bm{a}^{(0)}\neq\bm{0}$ and $f(\beta)=\sigma_0$ if $\bm{a}^{(0)}=\bm{0}$
(also see \cite[Formula (9)]{ZaJu13}). This shows the claim.

{\em Step 3.} The density matrix $N=n_0\sigma_0+\eps\bm{n}\cdot\bm{\sigma}$
equals $N^{(0)}=n_0\sigma_0$ at leading order, and the moment constraints are 
$\langle\M_0^{(0)}(N)\rangle=n_0^{(0)}$, $\langle\M^{(0)}(N)\rangle=0$
at leading order. Then equation \eqref{a.M0} shows that
$$
  0 = |\langle\M^{(0)}(N)\rangle|
	= \bigg|\big\langle\exp(h_0^{(0)})\big\rangle\sinh|\bm{a}^{(0)}|
	\frac{\bm{a}^{(0)}}{|\bm{a}^{(0)}|}\bigg|
	= \big|\big\langle\exp(h_0^{(0)})\big\rangle\sinh|\bm{a}^{(0)}|\big|,
$$
and it follows from $\langle\exp(h_0^{(0)})\rangle = 2\pi\exp(a_0^{(0)})\neq 0$
that $\bm{a}^{(0)}=\bm{0}$. This means that $\bm{a}$ vanishes at leading order,
and we can redefine the Lagrange multiplier matrix as 
$\tilde{A}=\tilde{a}_0\sigma_0+\eps\bm{a}\cdot\bm{\sigma}$.
This finishes the proof of Theorem \ref{thm.QM}.


\section{Semiclassical expansion of $g(\beta)$}\label{app.g}

In this section, we show formulas \eqref{g1}--\eqref{g3} for the orders
$g^{(j)}(\beta)$, where $j=1,2,3$. 
We use the notation $h_0=-|p|^2/2+a_0$ and $\bm{h}_1=\bm{a}-\alpha p^\perp$.

\subsection{Order one}

According to \eqref{4.aux}, the function $g^{(1)}(\beta)$ is the solution 
to the differential equation
$$
  \pa_\beta g^{(1)} = h_0\sigma_0 g^{(1)} + h_0\sigma_0\#_1 g^{(0)}
	+ (\bm{h}_1\cdot\bm{\sigma})g^{(1)}, \quad \beta>0, \quad g^{(1)}(0)=0.
$$
Since $g^{(0)}$ is a function of $h_0$, the Moyal product $h_0\#_1 g^{(0)}$
vanishes. Duhamels's formula then leads to the solution
$g^{(1)}(\beta)=\beta g^{(0)}\bm{h}_1\cdot\bm{\sigma}$, which equals \eqref{g1}.

\subsection{Order two}

The differential equation reads here as
$$
  \pa_\beta g^{(2)} = h_0\sigma_0\#_0 g^{(2)} + h_0\sigma_0\#_1 g^{(1)}
	+ h_0\sigma_0\#_2 g^{(0)} + (\bm{h}_1\cdot\bm{\sigma})\#_0 g^{(1)}
	+ (\bm{h}_1\cdot\bm{\sigma})\#_1 g^{(0)}, 
$$
with initial condition $g^{(2)}(0)=0$. The first term on the right-hand side
contains the unknown, while the others are known from the preceding orders.
Because of \eqref{moyal.two}, the $j$th Pauli component ($j=1,2,3$) of the second term 
can be written as
\begin{align*}
  2\ii h_0\#_1 g_j^{(1)} &= \na_p h_0\cdot\na_x
	\big(\beta g^{(0)}(a_j-\alpha p_j^\perp)\big)
	- \na_x h_0\cdot\na_p\big(\beta g^{(0)}(a_j-\alpha p_j^\perp)\big) \\
	&= -\beta g^{(0)}p\cdot\na_x a_j - \alpha\beta g^{(0)}\na_x a_0\cdot\na_p p_j^\perp \\
	&= \na_x(\bm{h}_1)_j\cdot\na_p g^{(0)} - \na_p(\bm{h}_1)_j\cdot\na_x g^{(0)}
	= -2\ii(\bm{h}_1)_j\#_1 g^{(0)}.
\end{align*}
Therefore, the terms $h_0\sigma_0\#_1 g^{(1)}$ and 
$(\bm{h}_1\cdot\bm{\sigma})\#_1 g^{(0)}$ cancel out. 
Since $\pa_x^\mu\pa_p^\nu h_0=0$ for multiindices satisfying 
$|\mu|\ge 1$ and $|\nu|\ge 1$, an elementary computation shows that
\begin{align*}
  h_0\#_2 g_0^{(0)} &= -\frac14\sum_{|\mu|+|\nu|=2}\frac{(-1)^{|\mu|}}{\mu!\nu!}
	(\pa_x^\mu\pa_p^\nu h_0)(\pa_p^\mu\pa_x^\nu g^{(0)}) 
	= -\frac18\sum_{i,k=1}^2\big(\pa^2_{x_i x_k}h_0\pa^2_{p_i p_k}g^{(0)}
	- \delta_{ik}\pa^2_{x_i x_k}g^{(0)}\big) \\
	&= \frac{\beta}{8}g^{(0)}\big[2\Delta a_0 - \beta\big(p^T(\na\otimes\na a_0)p
	- |\na a_0|^2\big)\big].
\end{align*} 
The final product $(\bm{h}_1\cdot\bm{\sigma})\#_0 g^{(1)}$ is just a multiplication.
We apply rule \eqref{2.pauli} with $a_0=b_0=0$ to obtain
$$
  (\bm{h}_1\cdot\bm{\sigma})\#_0 g^{(1)} = (\bm{h}_1\cdot\bm{\sigma})
	(\beta g^{(0)}\bm{h}_1\cdot\bm{\sigma}) = \beta g^{(0)}|\bm{h}_1|^2\sigma_0.
$$
Therefore, the differential equation for $g^{(2)}$ becomes
$$
  \pa_\beta g^{(2)} = h_0\sigma_0 g^{(2)}
	+ \beta g^{(0)}\bigg(\frac14\delta a_0 - \frac{\beta}{8}
	\big(p^T(\na\otimes\na a_0)p - |\na a_0|^2\big) + |\bm{h}_1|^2\bigg)\sigma_0
$$
for $\beta>0$ with initial datum $g^{(2)}(0)=0$. Duhamel's formula leads to
\eqref{g2}.

\subsection{Order three}

We need to solve the differential equation
\begin{equation}\label{a.dg3}
  \pa_\beta g^{(3)} = h_0\sigma_0 g^{(3)} + h_0\sigma_0\#_1 g^{(2)}
	+ h_0\sigma_0\#_2 g^{(1)} + h_0\sigma_0\#_3 g^{(0)}
	+ (\bm{h}_1\cdot\bm{\sigma})\#_1 g^{(1)} + (\bm{h}_1\cdot\bm{\sigma})\#_2 g^{(0)}
\end{equation}
for $\beta>0$ with initial datum $g^{(3)}(0)=0$.
To this end, we compute the right-hand side term by term. 
Since some of the computations are quite involved but straightforward,
we only report the results. It turns out that all $\sigma_0$-components cancel out
and only the $\bm{\sigma}$-components remain.

We write $g^{(2)}=(\beta^2/8)h_0 \lambda\sigma_0$, where
$$
  \lambda := \Delta a_0 + \frac{\beta}{3}\big(|\na_x a_0|^2 
	- p^T(\na_x\otimes\na_x a_0)p\big) + 4|\bm{h}_1|^2.
$$
Since $g^{(0)}$ is a function of $h_0$, we find for the second term that
\begin{align*}
  2\ii &h_0\#_1 g^{(2)} = \frac{\ii\beta^2}{4}g^{(0)}(h_0\#_1\lambda)\sigma_0
	= \frac{\beta^2}{8}h_0(\na_p h_0\cdot\na_x\lambda - \na_x h_0\cdot\na_p\lambda)
	\sigma_0 \\
	&= \frac{\beta^2}{8}g^{(0)}\bigg(\frac{\beta}{3}p\cdot\na_x(p^T(\na_x\otimes\na_x
	a_0)p) - p\cdot\na_x\Delta_x a_0 + 8\big(\alpha\na_x^\perp a_0 - (\na_x\bm{a})p\big)
	\cdot\bm{h}_1\bigg)\sigma_0.
\end{align*}
The next term reduces to $h_0\#_2\bm{g}^{(1)}\cdot\bm{\sigma}$, so we have to
calculate $h_0\#_2\bm{g}^{(1)}$:
\begin{align*}
  h_0\#_2\bm{g}^{(1)} &= -\frac14\sum_{|\mu|+|\nu|=2}\frac{(-1)^{|\mu|}}{\mu!\nu!}
	(\pa_x^\mu\pa_p^\nu h_0)(\pa_p^\mu\pa_x^\nu \bm{g}^{(1)}) \\
	&= -\frac18\sum_{i,k=1}^2\pa_{x_i x_k}^2h_0\pa_{p_ip_k}\bm{g}^{(1)}
	+ \frac18\sum_{k=1}^2\pa_{x_k x_k}^2\bm{g}^{(1)} \\
	&= \frac{\beta}{8}g^{(0)}\big[\big(2\beta\Delta_x a_0 + \beta^2(|\na_x a_0|^2
	- p^T(\na_x\otimes\na_x a_0)p)\big)\bm{h}_1 \\
	&\phantom{xx}{}+ 2\beta\alpha\na_x^\perp(\na_x a_0\cdot p)
	+ 2\beta\na_x\bm{a}\cdot\na_x a_0 + \Delta_x\bm{a}\big].
\end{align*}
We compute the fourth term on the right-hand side of \eqref{a.dg3} by observing 
that $\pa_x^\mu\pa_p^\nu h_0=0$ for $|\mu|\le 2$ and $|\nu|=3-|\mu|$:
\begin{align*}
  h_0\sigma_0\#_3 g^{(0)} &= \frac{1}{8\ii}\sum_{|\mu|=3}\frac{1}{\mu!}
	\pa_x^\mu h_0\pa_p^\mu g^{(0)}\sigma_0 \\
	&= \frac{\beta^2}{16\ii}g^{(0)}\bigg(p\cdot\na_x\Delta_x a_0
	- \frac{\beta}{3}p\cdot\na_x(p^T(\na_x\otimes\na_x a_0)p)\bigg)\sigma_0.
\end{align*}
The fifth term is just an ordinary multiplication between
$\bm{h}_1$ and $g^{(2)}$:
$$
  \bm{h}_1\#_0 g^{(2)} = \frac{\beta^2}{8}g^{(0)}\bigg(\Delta a_0 
	+ \frac{\beta}{3}\big(|\na_x a_0|^2 - p^T(\na_x\otimes\na_x a_0)p\big) 
	+ 4|\bm{h}_1|^2\bigg)\bm{h}_1\cdot\bm{\sigma}.
$$

The computation of the sixth term is a bit more involved. Formula \eqref{moyal.AB}
gives
\begin{equation}\label{a.aux}
  (\bm{h}_1\cdot\bm{\sigma})\#_1 g^{(1)}
	= (\bm{h}_1\cdot\bm{\sigma})\#_1(\bm{g}^{(1)}\cdot\bm{\sigma})
	= (\bm{h}_1\cdot_{\#_1}\bm{g}^{(1)})\sigma_0
	+ \ii(\bm{h}_1\times_{\#_1}\bm{g}^{(1)})\cdot\bm{\sigma},
\end{equation}
recalling that ``$\cdot_{\#_1}$'' and ``$\times_{\#_1}$'' are the usual vector
operations, where the multiplication is replaced by the order-one Moyal product.
Since $\bm{g}^{(1)}$ is a function of $\bm{h}_1$, it follows from \eqref{g1} that
\begin{align*}
  2\ii\bm{h}_1\cdot_{\#_1}\bm{g}^{(1)}
	&= \beta\big(\na_p\bm{h}_1\cdot\na_x(g^{(0)}\bm{h}_1) 
	- \na_x\bm{h}_1\cdot\na_p (g^{(0)}\bm{h}_1)\big) \\ 
	&= \beta^2 g^{(0)}\big(-\alpha\na_x^\perp a_0 + (\na_x\bm{a})p\big)
	\cdot\bm{h}_1.
\end{align*}
For the second term on the right-hand side of \eqref{a.aux}, we write
$2\ii\bm{h}_1\times_{\#_1}\bm{g}^{(1)}=\na_p\bm{h}_1\times\na_x\bm{g}^{(1)}
- \na_x\bm{h}_1\times\na_p\bm{g}^{(1)}$, where the cross product refers to the
vectors $\bm{h}_1$ and $\bm{g}^{(1)}$ and not to the gradients. Then, inserting
$g^{(1)}=\beta g^{(0)}\bm{h}_1\cdot\bm{\sigma}$ (see \eqref{g1} again), 
a computation shows that
$$
  2\ii\bm{h}_1\times_{\#_1}\bm{g}^{(1)} = \beta g^{(0)}\big[\beta\big(\na_x\bm{a} p
	- \alpha\na_x^\perp a_0\big)\times\bm{h}_1 - 2\alpha\na_x^\perp\times\bm{a}\big].
$$
Thus, \eqref{a.aux} becomes
\begin{align*}
  (\bm{h}_1\cdot\bm{\sigma})\#_1 g^{(1)}
	&= \beta g^{(0)}\bigg[\frac{\beta}{2\ii}\big((\na_x\bm{a})p - \alpha\na_x^\perp a_0
	\big)\cdot\bm{h}_1\sigma_0 \\
	&\phantom{xx}{}+ \bigg(\frac{\beta}{2}\big((\na_x\bm{a})p - \alpha\na_x^\perp a_0
	\big)\times\bm{h}_1 - \alpha\na_x^\perp\times\bm{a}\bigg)\cdot\bm{\sigma}\bigg].
\end{align*}
Finally, the last term on the right-hand side of \eqref{a.dg3}
is computed according to
$$
  \bm{h}_1\#_2 g^{(0)} = -\frac18\sum_{i,k=1}^2\pa^2_{x_i x_k}\bm{h}_1
	\pa^2_{p_i p_k}g^{(0)} = -\frac{\beta}{8}g^{(0)}\big[\beta p^T(\na_x\otimes\na_x
	\bm{a})p - \Delta_x\bm{a}\big].
$$

Substituting these expressions in \eqref{a.dg3}, we see that the 
$\sigma_0$-components cancel out, and we end up with the differential equation
\begin{align*}
  \pa_\beta g^{(3)} &= h_0\sigma_0 g^{(3)} + \frac{\beta}{8}\bigg[\beta\bigg(
	3\Delta_x a_0 + \frac43\beta\big(|\na_x a_0|^2 - p^T(\na_x\otimes\na_x a_0)p\big)
	+ 4|\bm{h}_1|^2\bigg)\bm{h}_1 \\
	&\phantom{xx}{}+ 2\Delta_x\bm{a} - 8\alpha\na_x^\perp\times\bm{a}
	+ \beta\big(2\na_x\bm{a}\cdot\na_x a_0 - p^T(\na_x\otimes\na_x\bm{a})p
	+ 2\alpha\na_x^\perp(\na_x a_0\cdot p)\big) \\
	&\phantom{xx}{}+ 4\beta\big((\na_x\bm{a})p - \alpha\na_x^\perp a_0\big)
	\times\bm{h}_1\bigg]\cdot\bm{\sigma}
\end{align*}
for $\beta>0$ with initial datum $g^{(3)}(0)=0$. Duhamels's formula then leads to
\eqref{g3}.

\end{appendix}


\end{document}